\documentclass[11pt]{amsart}
\usepackage{amsmath}
\usepackage{hyperref}
\usepackage[active]{srcltx}
\usepackage{xypic}
\usepackage{amscd}
\usepackage{color}
\newtheorem{thm}{Theorem}[section]
\newtheorem{cor}[thm]{Corollary}
\newtheorem{lemma}[thm]{Lemma}

\newtheorem{proposition}[thm]{Proposition}

\theoremstyle{definition}
\newtheorem{remark}[thm]{Remark}

\newtheorem{definition}[thm]{Definition}
 
 \newtheorem{example}[thm]{Example}

\DeclareMathOperator{\im}{Im}

\DeclareMathOperator{\Sing}{Sing}
\DeclareMathOperator{\Aut}{Aut}
\DeclareMathOperator{\Ext}{Ext}

\DeclareMathOperator{\Pic}{Pic}

\def\max{\operatorname{max}}

\def\c1{\operatorname{c_1}}
\def\c2{\operatorname{c_2}}
\def\Cliff{\operatorname{Cliff}}

\def\rk{\operatorname{rk}}

\def\s{\mathfrak{s}}
\def\f{\mathfrak{f}}
\def\n{\mathbf{N}}

\def\c{\mathfrak{P}}
\def\CC{{\mathbb C}}
\def\ZZ{{\mathbb Z}}

\def\DD{{\mathbb D}}

\def\PP{{\mathbb P}}

\def\R{{\mathcal R}}
\def\L{{\mathcal L}}
\def\M{{\mathcal M}}
\def\N{{\mathcal N}}
\def\O{{\mathcal O}}
\def\I{{\mathcal J}}
\def\D{{\mathcal D}}
\def\Z{{\mathcal Z}}
\def\E{{\mathcal E}}

\def\H{{\mathcal H}}
\def\F{{\mathcal F}}

\def\V{{\mathcal V}}

\def\C{{\mathcal C}} 

\def\X{{\mathcal X}}
\def\P{{\mathcal P}}

\def\FF{{\mathbb F}}
\def\x{\times}                   
\def\cong{\simeq}

\def\+{\oplus}                   
\def\*{\otimes}                  

\def\mod{\operatorname{mod}}
\def\Aut{\operatorname{Aut}}

\def\Ext{\operatorname{Ext}}

\def\hom{\operatorname{ \mathfrak{hom}}}
\def\ext{\operatorname{ \mathfrak{ext}}}

\def\Pic{\operatorname{Pic}}

\def\Sing{\operatorname{Sing}}

\begin{document}

\title[Half Nikulin surfaces and moduli of Prym curves]{Half Nikulin surfaces and moduli of Prym curves}

\author[A.~L.~Knutsen]{Andreas Leopold Knutsen}
\address{Andreas Leopold Knutsen, Department of Mathematics, University of Bergen,
Postboks 7800,
5020 Bergen, Norway}
\email{andreas.knutsen@math.uib.no}

\author[M.~Lelli-Chiesa]{Margherita Lelli-Chiesa}
\address{Margherita Lelli-Chiesa, Dipartimento di Matematica,
Universit{\`a} di Pisa, Largo Pontecorvo 5, 56127 Pisa, Italy}
\email{m.lellichiesa@gmail.com} 

\author[A.~Verra]{Alessandro Verra}
\address{Alessandro Verra, 
Dipartimento di Matematica,
Universit{\`a} Roma Tre,
Largo San \linebreak Leonardo Murialdo,  
00146 Roma,  Italy} \email{verra@mat.uniroma3.it}

\begin{abstract} 
Let $\F^{\n}_g$ be the moduli space of polarized Nikulin surfaces 
$(Y,H)$ of genus $g$ and let $\P^{\n}_g$ be the moduli of triples
$(Y,H,C)$, with $C \in \vert H \vert$ a smooth curve. We study the natural map $\chi_g: \P^{\n}_g \to \mathcal R_g$, where $\R_g$ is the moduli space of Prym curves of genus $g$. 
 We prove that it is generically injective on every irreducible component, with a few exceptions in low genus.  This gives a complete
 picture of the map $\chi_g$ and confirms some striking analogies between it and the Mukai map $m_g: \P_g \to \mathcal M_g$ for moduli of triples $(Y,H,C)$, where $(Y,H)$ is any genus $g$ polarized  $K3$ surface.
 The proof is by degeneration to boundary points of a partial compactification of $\F^{\n}_g$. These represent the union of two surfaces with four even nodes and effective anticanonical class,
 which we call half Nikulin surfaces. The use of this degeneration is new with respect to previous techniques.
\end{abstract}

\maketitle


\section{Introduction} \label{sec:intro}
Complex projective $K3$ surfaces have been an object of study since many years before the birth of their name. These surfaces were indeed investigated by classical algebraic geometers both from the point of view of their automorphisms and of their projective models. Later, for important historical reasons, $K3$ surfaces started to play a central role in several branches of algebraic geometry and not only. 

The modern study of linear systems on $K3$ surfaces, initiated by Saint-Donat in the seventies \cite{SD}, paved the way to important results in the theory of algebraic curves. 
This trend is well represented by Green's conjecture and  the timeline of recent results, leading to the conclusive proof of the conjecture in the case of smooth curves  on any $K3$ surface \cite{Vo1,Vo2,AF}.  Furthermore, in modern times, the role of  $K3$ sections  in the study of the birational geometry of the moduli space $\mathcal M_g$ of genus $g$ curves has become well-established.  We mention the pioneering work of Mori and Mukai on the uniruledness of $\mathcal M_{11}$ and Mukai's realizations of canonical curves in low genus \cite{MM,muk2}.    

Let $\mathcal F_g$ be the moduli space of primitively polarized $K3$ surfaces $(Y,H)$ of genus $g$ (that is, $H \in \Pic Y$ is primitive, big and nef and $H^2=2g-2$), and  $\mathcal P_g$ be the moduli space of triples $(Y,H,C)$ with $(Y,H) \in \F_g$ and $C \in \vert H \vert$ a smooth curve. 
There are natural forgetful morphisms
\begin{equation} \label{eq:forgetful}
\xymatrix{ 
&\P_g\ar_{q_g}[ld]\ar^{m_g}[rd]&\\
\F_g&&\M_g.
}
\end{equation}
The study of the {\it Mukai map} $m_g$ is a chapter of the history we are outlining and, indeed, a motivation for this paper.

Let 
$ 
\Pic_{d,g} \to \mathcal M_g
$
 be the universal Picard variety, whose fibre over 
$C$ is $\Pic^dC$. For each $n>0$, the assignment  $(Y,H,C) \to \mathcal O_C(nH)$ defines a lifting $\mathcal P_g \to \Pic_{2n(g-1),g}$ of $m_g$. There is no hope that other liftings exist; the reason is essentially that $\Pic Y$ is generated by $H$ for a very general pair $(Y ,H)$. However, 
liftings may exist over proper subloci of $\F_g$ that parametrize surfaces carrying  other line bundles; this has already been employed to address very interesting cases (cf., e.g., \cite{FK}).

In this paper we investigate some liftings of the Mukai map in order to understand the natural relations between the moduli space of Prym curves of genus $g$
$$
\mathcal R_g \subset \Pic_{0,g}
$$
and the moduli of those $K3$ surfaces, suitably polarized in genus $g$, which are quotients of $K3$ surfaces endowed with a symplectic involution. 
Before presenting our results,
we  add a few words on these moduli spaces, revisiting some basic facts and definitions. 

A {\it Prym curve} of genus $g$ is a pair $(C, \eta)$ such that $C$ is a smooth  curve of genus $g$ and $\eta \in \Pic^0 C$ is a nontrivial $2$-torsion element. The corresponding moduli space is denoted by $\mathcal R_g$. The above $K3$ surfaces, or more precisely their minimal desingularizations, are known as {\it Nikulin surfaces}, due to Nikulin's classification of symplectic automorphisms of $K3$ surfaces, a part of his foundational work on $K3$ surfaces \cite{Ni}.
Before giving their definition, we provide an example of a special family of such surfaces, which is also useful to introduce another actor of this paper. 
Let $\iota$ be an involution of $\PP^1 \times \PP^1$ with exactly $4$ fixed points, then $\overline{X}:= \PP^1 \times \PP^1 / \langle \iota \rangle$ is a $4$-nodal Del Pezzo surface of degree $4$. Let $q: \PP^1 \times \PP^1 \to \overline{X}$ be the quotient map. Let $B \in \vert \omega_{\overline{X}}^{-2} \vert$ be a smooth anti-bicanonical curve and $p: \overline{Y } \to {\overline X}$ the double cover branched on it.  Consider the Cartesian square
$$
\xymatrix{ \widetilde{Y} \ar[r]^{\tilde{q}} \ar[d]_{\tilde{p}} & \overline{Y} \ar[d]^p \\
\PP^1 \x \PP^1 \ar[r]^{q} & \overline{X}} 
$$
It turns out that $\widetilde{Y}$ is a $K3$ surface with a symplectic involution and that its singular quotient $\overline{Y }$ is a Nikulin surface with $8$ nodes.
 In this paper $\overline X$, or more precisely its minimal desingularization,  is an example of what we call a {\it half Nikulin surface}. These surfaces are 
 crucial for our work.  Indeed, we will use the gluing of two half Nikulin surfaces along an anticanonical curve for our degeneration arguments.

A Nikulin surface is a smooth $K3$ surface $Y$ endowed with a nontrivial double cover $\pi:\widehat Y\to Y$, whose branch divisor consists of eight disjoint smooth rational curves $N_1,\ldots,N_8$. Set $M:=\frac{1}{2}\O_Y(N_1+\cdots+N_8)$. If $H$ is a big and nef line bundle on $Y$ with $H^2=2(g-1)$ and $H\cdot N_i=0$ for $i=1,\ldots,8$, the triple $(Y,M,H)$ is called a {\it polarized Nikulin surface of genus $g$}. For any smooth curve $C\in |H|$ the restriction of $\pi$ to $\pi^{-1}(C)$ defines an \'etale double cover of $C$; in other words, the pair $(C,M\otimes \O_C)$ defines a point of $\R_g$.  Thus, one obtains a lifting of the Mukai map over the locus of polarized Nikulin surfaces of genus $g$. 

The Picard group of any Nikulin surface contains a lattice isomorphic to $\Lambda_g := \mathbb Z[H] \perp \mathbf N$, where $\mathbf N$ denotes the rank eight {\it Nikulin lattice} generated by $N_1,\ldots, N_8$ and $M$.
Using Dolgachev's theory of lattice-polarized $K3$ surfaces \cite{Do}, Sarti and van Geemen in \cite{vGS} and Garbagnati and Sarti in \cite{GS} have shown that 
primitively polarized  Nikulin surfaces are of two types, according to whether the embedding $\Lambda_g \subset \Pic Y$ is primitive or not; we will refer to the two types as {\it standard} and {\it non-standard}  (cf. \S\ref{ss:def}), the latter occuring only in odd genera.   There are coarse moduli spaces $\F_g^{\n,s}$ and $\F_g^{\n,ns}$ parametrizing genus $g$  primitively  polarized Nikulin surfaces of standard and non-standard type, respectively. They are both irreducible of dimension $11$, cf. \cite[\S 3]{Do}, \cite[Prop. 2.3]{vGS}. 

We denote by $\P_g^{\n,s}$  the 
restriction of $\P_g$ over $\F_g^{\n,s}$, and by 
\begin{equation}\label{nikulin}
\xymatrix{ 
&\P_g^{\n,s} \ar[ld]_{q_g^{\n,s}}  \ar[d]_{\chi_g^{s}}    \ar[rd]^{m_g^{\n,s}} &\\
\F_g^{\n,s}&   \R_g \ar[r]_{p_g}    &\M_g
}
\end{equation}
 the restriction of \eqref{eq:forgetful} in the standard case, replacing ``$s$'' by ``$ns$'' in the nonstandard case;
here, $\chi_g^{s}$ is the above mentioned lifting of the Mukai map
applying $((Y,M,H),C)$ to $(C, M\otimes \O_C)$
and $p_g$ is the  forgetful covering map of degree $2^{2g}-1$.

The behavior of these maps can be interestingly compared with the behavior of the Mukai map $m_g$:
\begin{itemize}
\item[(i)] $m_g$ is dominant for $g \leq 11$  and $g\neq 10$ \cite{muk2};\  
\item[(ii)] $m_{11}$ is birational  \cite{muk1}; 
\item[(iii)] the image of $m_{10}$ is a divisor in $\M_{10}$ \cite{muk2}; 
\item[(iv)] $m_{g}$ is birational onto its image for $g \geq 11$ and $g\neq 12$ \cite{MM,clm};
\item[(v)] $m_{12}$ has generically one-dimensional fibers \cite{muk3}.
\end{itemize}    
 To enrich the picture recall that the slope conjecture is false for the image of $m_{10}$ \cite{FP}.  Also note that
the refined study of $m_g$ in higher genus is presently very intense: Mukai's program towards reconstructing a fibre of $m_g$ is now proven for $g\equiv 3 \,\mathrm{mod}\, 4$ with $g\geq 11$ \cite{abs1, muk1}. Moreover, the image of $m_g$ has been recently characterized, via the Gaussian map, for Brill-Noether-Petri general curves \cite{abs, W2}. 

The study of the map $\chi_g^{s}$ was started by Farkas and the  third named author  in \cite{FV}. Already in low genera, $\chi_g^{s}$  offers unexpected and interesting analogies to the Mukai map.  The turning point is 
here genus $7$ and not $11$; let us quote from \cite{FV}: 
 \begin{enumerate} {\it
\item $\chi_g^{s}$ is dominant for $g \leq 7$  and $g\neq 6$.
\item the image of $\chi_6^{s}$ is a divisor.
\item $\chi_7^{s}$ is birational.}
 \end{enumerate}    
The image of $\chi^s_6$ is again an interesting divisor, as it is the ramification locus of the Prym map $P_6: \mathcal R_6 \to \mathcal A_5$, and its role is crucial in computing
the slope of a suitable compactification of the moduli space $\mathcal A_5$ of principally polarized abelian $5$-folds  \cite[Thm. 0.5]{FV}.  
The first main result of this paper completes the picture of the map $\chi^s_g$:
 
\begin{thm} \label{intro-thm:main-c}
The map $\chi_g^{s}$ is birational onto its image if $g\geq 7$ and $g\neq 8$, 
while its general fiber is a rational curve for $g = 8$. 
 \end{thm}
Hence, we retrieve the analogues of the properties (iv) and (v) of the Mukai map.  

A major difference between the standard and nonstandard case is the Brill-Noether behaviour of general Nikulin sections. Indeed, a general curve in the image of $m_g^{\n,s}=p_g \circ \chi^s_g$ is Brill-Noether-Petri general (cf. Proposition \ref{prop:cliff}), while a general curve in the image of $m_g^{\n,ns}=p_g \circ \chi^{ns}_g$ carries two distinguished theta-characteristics that make it quite special in moduli (cf. Remark \ref{rem:extheta}). As a consequence, $\chi_g^{ns}$ can never be dominant. Furthermore, a heuristic count suggests that it cannot be generically finite for $g=9$ and $11$. The second main result of this paper proves that the situation is as nice as possible:
 
\begin{thm} \label{intro-thm:main-nc}
The map $\chi_g^{ns}$ is birational onto its image for (odd) genus $g\geq 13$.
\end{thm}

In a forthcoming paper \cite{klv2}, we will show that the above theorem is optimal, as a general fiber of $\chi^{ns}_g$ has dimension four if $g=7$, two if $g=9$ and one if $g=11$. 

It is also natural to pose the question of the 
degree of (the Stein factorization of) the maps $m_g^{\n,s}$ and $m_g^{\n,ns}$.  Since $\deg p_g  >1$, the degree cannot be one if 
$\R_g$  is dominated, that is, if $g \leq 7$ with $g \neq 6$ in the standard case.
Otherwise we expect that the Stein factorization of $m_g^{\n,s}$ or $m_g^{\n,ns}$ has degree one, that is, the map either has degree one or has positive
dimensional connected fibers. More geometrically, we expect that all Nikulin surfaces containing a general $C$ define the same \'etale double cover of $C$.
Our next result offers a quite positive answer.
 \begin{thm} \label{intro-thm:f}  
The map $m_g^{\n,s}$ is birational onto its image for $g\geq 11$ and $g\not\in\{ 12,14\}$.

The map $m_g^{\n,ns}$ is birational onto its image for $g=13$ and 
(odd) $g\geq 17$. 
\end{thm} 

Degeneration methods are the core of the proofs of our theorems. We exploit degenerations to a particular class of type II $K3$ surfaces in the Kulikov-Persson-Pinkham classification \cite{Ku,pp}. After Friedman's partial compactification of $\mathcal F_g$ \cite{fri,fri2}, a number of type II $K3$ surfaces occurred in a variety of applications. We mention especially
the gluing of two Hirzebruch surfaces along a suitable section \cite{clm, ck,cfgk}.

The main technical achievement of our work is to provide, and use, type II $K3$ surfaces occurring as limits of {\it Nikulin} surfaces. The construction of these limits relies on {\it half Nikulin surfaces},
already present  in our example of Nikulin surfaces. These are smooth rational surfaces $X$ containing a smooth irreducible anticanonical curve $A$ and the sum $N$ of four disjoint rational curves,
so that $N$ or $N + A$ is $2$-divisible  in $\Pic X$. Accordingly, we call the half Nikulin surface of {\it untwisted} or {\it twisted type}.
The gluing of two half Nikulin surfaces along a smooth anticanonical curve yields a type II $K3$ surface that turns out to be a  limit of  Nikulin surfaces. It plays a central role in the proof of our results, as its very rich geometry enables us to reconstruct the surface starting from a hyperplane section of it.  As a byproduct, this degeneration
provides a new proof of the existence of an $11$-dimensional component of
$\F_g^{\n,s}$ and $\F_g^{\n,ns}$ that is purely algebro-geometric and does not
rely on any transcendental lattice-theoretical method. We do believe that the family of boundary $K3$ surfaces we have constructed is worth of further study.

\vspace{0.3cm}
\noindent {\it Organization of the paper.}
In \S \ref{sec:def} we recall the basic definitions and properties of Nikulin surfaces and explain the strategy of the proofs of our main results, that proceed by degeneration to boundary points of suitable partial compactifications of $\F_g^{\n,s}$ and $\F_g^{\n,ns}$. Proposition
\ref{prop:genere8} proves Theorem 1.3 in the exotic case of genus $8$. 

In  \S \ref{sec:halfnikulin}  we collect  general results on limits of $K3$ surfaces and their deformations. In  particular, Lemma \ref{lemma:utilissimo} is an essential tool to study deformations of hyperplane sections of such limits in a smoothing family of $K3$ surfaces. In \S \ref{ss:defhalf} {\it half $K3$ surfaces} and {\it half Nikulin surfaces of untwisted and twisted type} are introduced.   These can be reconstructed  from their hyperelliptic hyperplane sections, 
see Proposition \ref{prop:recon}.
Section \ref{sec:const-unt} exhibits the main series of examples of half Nikulin surfaces of untwisted type. These surfaces are used in \S \ref{sec:SS} to construct  boundary divisors in partial compactifications of both $\F_g^{\n,s}$ and $\F_g^{\n,ns}$,  cf. Corollary \ref{cor:exparcomI}. These compactifcations are exploited in the proofs  of the main theorems for  almost all genera in the standard case and for genera $g\equiv 1\, \mathrm{mod}\,4$ in the non-standard one. Different compactifications are constructed in \S\ref{sec:parcom2} starting from the same half Nikulin surfaces endowed with different polarizations. They allow to cover the non-standard case also for genera $g\equiv 3\, \mathrm{mod}\,4$. 

In \S \ref{sec:parcom} examples of half Nikulin surfaces of twisted type are produced by blowing up rational normal scrolls at four pairs of infinitely near points. These surfaces occur as components of degenerations of Nikulin surfaces of odd genus and standard type, used to establish the generic injectivity of the maps $\chi_g^{s}$ and $m_g^{\n,s}$  in the few cases left. 

More precisely, Theorems \ref{intro-thm:main-c} and \ref{intro-thm:main-nc} are consequences of Theorems \ref{thm:main-I}, 
\ref{thm:main-II} and \ref{thm:main-III}, while
Theorem \ref{intro-thm:f} follows from Theorems \ref{thm:main2-I}, 
\ref{thm:main2-II} and \ref{thm:main2-III}.

\vspace{0.3cm}
\noindent {\it Acknowledgements.} 
The authors wish to thank C.~Ciliberto, T.~Dedieu and E.~Sernesi for useful conversations.  The first author has been partially supported by grant n. ~261756 of the Research Council of Norway. 
The second and third named authors were supported by the Italian PRIN-2015 project 'Geometry of Algebraic varieties' and by GNSAGA.

\section{Nikulin surfaces and their moduli maps} \label{sec:def}

\subsection{Some definitions and properties}\label{ss:def}
We recall the following:
\begin{definition} \label{def:Nik}
  A (polarized)  {\it Nikulin} surface of genus $g \geq 2$ is a triple $(Y,M,H)$ such that 
$Y$ is a smooth $K3$ surface with $\O_Y(M),H \in \Pic Y$ satisfying
\begin{itemize}
\item $Y$ carries mutually disjoint rational curves $N_1,\ldots,N_8$ such that $\sum_{i=1}^8N_i\sim 2M$; 
\item $H$ is nef, $H^2=2(g-1)$ and  $H \cdot M=0$.
\end{itemize}
We say that $(S,M,H)$ is {\it primitively polarized} if in addition $H$ is primitive in $\Pic Y$.
\end{definition}

The line bundle $\mathcal O_Y(M)$  defines a double cover \mbox{$\pi:\widehat Y\to Y$} branched on $\sum_{i=1}^8N_i$. This fits into a Cartesian square:
\begin{equation}\label{nik}
 \xymatrix{
 {\widehat Y}  \ar[r]^{\hat \tau} \ar[d]_{\pi}&   {\widetilde Y} \ar[d]^{\bar{\pi}} \\
 Y  \ar[r]^{\tau} & {\overline Y},}
 \end{equation}
  where $\tau$ and $\hat\tau$ are the contractions of the curves $N_i$ and of their inverse images on $\widehat Y$, respectively.  Since the latter are $(-1)$-curves, the surface $\widetilde Y$ is  a smooth $K3$ surface endowed with an involution $\iota$ with exactly $8$ fixed points. The map $\overline \pi$ in \eqref{nik} is the quotient of $\widetilde{Y}$ by $\iota$ and thus $\overline{Y}$ has $8$ double points.

\begin{definition} \label{def:Nik2}
The {\it Nikulin lattice} $\mathbf{N}=\mathbf{N}(Y,M)$ of a genus $g$ Nikulin surface $(Y,M,H)$  is the rank $8$ sublattice of $\Pic Y$ generated by $N_1,\ldots,N_8$ and $M$.

A { \it primitively polarized \rm } Nikulin surface $(Y,M,H)$ is {\it standard} if the embedding of the 
rank $9$ lattice
\[ \Lambda=\Lambda(Y,M,H):= \ZZ[H] \+_{\perp} \mathbf{N} \subset \Pic Y\]
 is primitive, and {\it non-standard} otherwise.
\end{definition} 

By \cite[Prop. 2.1, Cor. 2.1]{GS}, in the non-standard case the embedding $\Lambda\subset\Pic Y$ has index $2$ and the genus $g$ is odd.  Moreover, possibly after renumbering the curves $N_i$, the following classes  $v, v' \in \Pic Y$ are $2$-divisible: 
\begin{itemize}
\item $v= H-N_1-N_2-N_3-N_4$ and $v'=H-N_5-N_6-N_7-N_8$, if $g \equiv 1 \, \mod 4$; 
\item $v=H-N_1-N_2$ and $v'=H-N_3-\cdots-N_8 $, if $g \equiv 3 \, \mod 4$.
\end{itemize}
Furthermore, standard (respectively, non-standard) Nikulin surfaces of Picard rank $9$ exist in any genus (resp.,  any odd genus), cf. \cite[Prop.~2.3]{GS}.

The next result is along the same lines as \cite[Thm.~0.5]{AF1} in the standard case.

\begin{proposition} \label{prop:cliff}
  Let $(Y,M,H)$ be a primitively polarized Nikulin surface of genus $g$ such that $\rk \Pic Y=9$. 

If $Y$ is standard, then all smooth curves in $|H|$ are Brill-Noether general, and the general ones are Brill-Noether-Petri general. 

If $Y$ is non-standard and $g \equiv 1 \; \mod 4$ (respectively, $g \equiv 3 \; \mod 4$), then any smooth curve $C$ in $|H|$ has Clifford index $\frac{g-1}{2}$ (resp., $\frac{g-3}{2}$).
\end{proposition}

\begin{proof}
If $Y$ is standard, then $\Pic Y \cong \Lambda$ and one may check that there is no decomposition $H \cong H_1\otimes H_2$ in $\Pic Y$ with $h^0(H_i) \geq 2$ for $i=1,2$. As in \cite{La}, one shows that
any smooth curve in $|H|$ satisfies the Brill-Noether Theorem, and a general one also fulfills the Gieseker-Petri Theorem. 

Assume that $Y$ is non-standard. By \cite{gl}, all smooth curves in $|H|$ have the same Clifford index $c$. Moreover, if $c < (g-1)/2$,  there is a decomposition $H \cong H_1\otimes H_2$ in $\Pic Y$ with $h^0(H_i) \geq 2$ for $i=1,2$ such that 
$c=\Cliff C=\Cliff\O_C(H_1)=H_1 \cdot H_2-2$ for any smooth $C \in |H|$, cf. \cite{jk,kn-manus}. 
Conversely, for any decomposition $H \cong H_1\otimes H_2$ in $\Pic Y$ with $h^0(H_i) \geq 2$ for  $i=1,2$, the line bundles $\O_C(H_i)$ contribute to the Clifford index and $\Cliff\O_C(H_i)=H_1\cdot H_2-2 \geq c$. Therefore, to compute $c$ we have to search for decompositions of $H$ as above with {\it minimal} $H_1 \cdot H_2$. This is an easy exercise using the fact  that $\Pic S \cong \ZZ[v/2] \+ \mathbf{N}$ by 
\cite[Prop. 2.1 and Cor. 2.1]{GS}.
\end{proof}

\begin{remark} \label{rem:extheta}
If $(Y,M,H)$ is a non-standard Nikulin surface of genus $g \equiv 1 \, \mod 4$ (respectively, $g \equiv 3 \, \mod 4$), any $C\in|H|$ carries two distinguished theta-characteristics, namely, $\mathcal O_C(\frac v2)$ and $\mathcal O_C(\frac{v'}2)$.
  They   satisfy $h^0\left(\O_C\left(v/2\right)\right))=
(g+3)/4$ (resp., $(g+5)/4$) and $h^0\left(\O_C\left(v'/2\right)\right))=
(g+3)/4$ (resp., $(g+1)/4$). 
 In particular, they prevent $C$ from being Brill-Noether general. Hence, the moduli maps $\chi_g^{ns}$ and $m_g^{\n,ns}$ can never be dominant. Furthermore, a heuristic count 
comparing the dimension of $\P_g^{\n,ns}$ with the expected dimension of the locus of curves in $\M_g$ carrying two theta-characteristics as above suggests that the generic fiber dimension of $m_g^{\n,ns}$ is $4$, $2$ and $1$ for $g=7,9,11$, respectively. This  expectation will be proved in \cite{klv2}. 
\end{remark}

\subsection{Moduli maps and the strategy of proof} \label{ss:strat}

We recall the definitions of the parameter spaces $\F_g^{\n,s}$ and $\P_g^{\n,s}$ and the maps $q_g^{\n,s}$, $m_g^{\n,s}$ and $\chi_g^{s}$ in the introduction, as well as their analogues in the non-standard case.    We are going to sketch the strategy of proof of the birational statements in the main theorems, concentrating on the standard case. 
We  first focus on Theorem 1.1. Thanks to the double cover \eqref{nik} associated with any Nikulin surface $(Y,M,H)$, the generic injectivity of $\chi_g^{s}$ is equivalent to the generic injectivity of the map $\widetilde{m}_{2g-1}^s$ in the following diagram:
\begin{equation}\label{involution}
\xymatrix{ 
&\widetilde\P_{2g-1}^s\ar_{\widetilde{q}_{2g-1}^s}[ld]\ar^{\widetilde{m}_{2g-1}^s}[rd]&\\
\widetilde{\F}_{2g-1}^s&&\M_{2g-1}\,\,.
}
\end{equation}
where $\widetilde{\F}_{2g-1}^s$ is the moduli space of primitively polarized $K3$ surfaces $(\widetilde{Y},\widetilde{H},\iota)$ of genus $2g-1$ with a Nikulin involution $\iota$ of  standard type (cf. \cite{vGS}), and $\widetilde\P_{2g-1}^s$ is the open subset of a $\PP^g$-bundle over $\widetilde{\F}_{2g-1}^s$ whose fiber over $(\widetilde{Y},\iota,\widetilde{H})$ consists of all smooth integral curves in $|\widetilde{H}|$ invariant  under $\iota$. The maps in \eqref{involution} are restrictions of the following ones:
\begin{equation}\label{K3}
\xymatrix{ 
&\P_{2g-1}\ar_{q_{2g-1}}[ld]\ar^{m_{2g-1}}[rd]&\\
\F_{2g-1}&&\M_{2g-1},
}
\end{equation}
where $\F_{2g-1}$ is the moduli space of genus $2g-1$ primitively polarized $K3$ surfaces $(\widetilde{Y},\widetilde{H})$ and $\P_{2g-1}$ is the open subset of a $\PP^{2g-1}$- bundle with fiber over $(\widetilde{Y},\widetilde{H})$ parametrizing all smooth integral curves in $|\widetilde{H}|$.  The map $m_{2g-1}$ is birational onto its image (cf. \cite{clm}) for $2g-1\geq 13$, that is, $g\geq 7$. The generic injectivity of  $\widetilde{m}_{2g-1}^s$ can be proved by showing that the fiber of $m_{2g-1}$ over a general $[\widetilde{\gamma}]\in \im \widetilde{m}_{2g-1}^s$ consists of only one point. By \cite{clm}, this follows if $\widetilde{\gamma}$ has a corank one Gaussian map cf. \cite[Sketch of  pf.  of Prop. 3.3]{LC}. On the other hand, since $\Cliff(\widetilde{\gamma})\geq 3$ (by Proposition \ref{prop:cliff} and \cite{BH}) and $2g-1\geq 11$, the fiber  $m_{2g-1}^{-1}([\widetilde{\gamma}])$ is positive dimensional as soon as the Gaussian map of $\widetilde{\gamma}$ has corank $>1$, cf. \cite[Thm. 3]{abs} and \cite[Thm. 7.1]{W}. Hence, to show that $m_{2g-1}^{-1}([\widetilde{\gamma}])$ consists of exactly one point, it suffices to prove that it is finite. 

We introduce partial compactifications $\overline\F_{2g-1}$ and  $\overline\P_{2g-1}$ of $\F_{2g-1}$ and $\P_{2g-1}$ respectively,  and extend \eqref{K3} to:
\begin{equation}\label{compactification}
\xymatrix{ 
&\overline\P_{2g-1}\ar_{\overline q_{2g-1}}[ld]\ar^{\overline m_{2g-1}}[rd]&\\
\overline\F_{2g-1}&&\overline\M_{2g-1}.\,\,
}
\end{equation}
The boundary of $\overline\F_{2g-1}$ parametrizes surfaces obtained by gluing two smooth irreducible rational surfaces along a smooth elliptic curve that is anticanonical on each. The numerical invariants of the two components will be fixed according to the genus and  type (standard  or non-standard) considered. 

The restriction of $\overline{q}_{2g-1}$ over $\F_{2g-1}$ coincides with $q_{2g-1}$ and its fiber over a reducible surface $(\widetilde{S},\widetilde{H})$ consists of curves $\widetilde{C}\in |\widetilde{H}|$ with only nodes as singularities, all of which lie on $\Sing \widetilde{S}$. 

We consider a general point $((\widetilde{S},\iota,\widetilde{H}),\widetilde{C})$ in the closure of $\widetilde\P_{2g-1}^s$ in $\overline\P_{2g-1}$ and study the fiber of $\overline{f}_{2g-1}$ over $[\widetilde{C}]$. If this is finite, we are done by upper semicontinuity. Unfortunately, in most cases this does not hold true. We circumvent this problem by considering the analogoue of \eqref{K3} at the Hilbert scheme level:
\begin{equation}\label{hilbert}
\xymatrix{ 
&\mathfrak{P}_{2g-1}\ar_{\mathfrak{q}_{2g-1}}[ld]\ar^{\mathfrak{m}_{2g-1}}[rd]&\\
\mathfrak{H}_{2g-1}&&\mathfrak{C}_{2g-1}\,\,.
}
\end{equation}
Here $\mathfrak{H}_{2g-1}$ denotes the component of the Hilbert scheme of degree  $4g-4$  surfaces in $\PP^{2g-1}$ containing smooth primitively embedded $K3$ surfaces of genus $2g-1$. The space $\mathfrak{P}_{2g-1}$ denotes the flag Hilbert scheme of pairs $\gamma\subset Y\subset\mathbb{P}^{2g-1}$ with $[Y\subset\mathbb{P}^{2g-1}]\in \mathfrak{H}_{2g-1}$ and $\gamma$ a hyperplane section of it, while $\mathfrak{C}_{2g-1}$ is the Hilbert scheme containing canonical curves of genus $2g-1$ in $\mathbb{P}^{2g-1}$ (each living in some hyperplane). The fibers of $\mathfrak{m}_{2g-1}$ have dimension at least $2g$, which is the dimension of the space of projectivities fixing a hyperplane. Since a fiber of $m_{2g-1}$ is the quotient of a fiber of $\mathfrak{m}_{2g-1}$ by the projective group, it is enough to show that for a general $[\widetilde{\gamma}]\in \im \widetilde{m}_{2g-1}^s$ the fiber of $\mathfrak{m}_{2g-1}$ over a point $[\widetilde{\gamma}\subset\PP^{2g-1}]\in \mathfrak{C}_{2g-1}$ has dimension $2g$. As above, we consider a general point $((\widetilde{S},\widetilde{H},\iota),\widetilde{C})$ in the closure of $\widetilde\P_{2g-1}^s$ in $\overline\P_{2g-1}$, along with the embedding $\widetilde C\subset \widetilde S\subset \PP^{2g-1}$ determined by the line bundle $\widetilde H$ (up to projectivities). It is then enough to show that a component of the fiber of $\mathfrak{m}_{2g-1}$ over $[\widetilde{C}\subset \PP^{2g-1}]$ has dimension $2g$.

The strategy of proof of  Theorem 1.2 for $m_g^{\n,s}$ is basically the same. To prove that $m_g^{\n,s}$ is birational onto its image, it suffices to show that the fiber $f_g^{-1}([C])$ over a general $[C] \in \im m_g^{\n,s}$ consists of only one point; as above, one reduces to showing that  $f_g^{-1}([C])$ is finite. Again this is done by degeneration considering the forgetful maps between Hilbert schemes as in \eqref{hilbert} but for genus $g$.

We end this section by proving Theorem 1.1 in the exceptional case $g=8$. 

\begin{proposition} \label{prop:genere8}
  A general fiber of the map $\chi_8^{\n,s}$ is a rational curve.
\end{proposition}
\begin{proof}
  As proved in \cite{Ve}, a {\it general} primitively polarized Nikulin surface $(Y,M,H)$ of genus $8$ is embedded in $\PP^6$ by the line bundle $H(-M)$ as
$$
Y  = Q \cap T =Q\cap (\PP^6 \cap \mathbb{G}(1,4)) \subset \PP^9,
$$
where $Q$ is a quadric hypersurface, $T$ is a smooth quintic Del Pezzo threefold, $\PP^6 \subset \PP^9$ is a six-dimensional linear 
subspace and
$\mathbb{G}(1,4)$ is the Pl{\"u}cker embedding of the Grassmannian of lines in 
$\PP^4$. Furthermore, $H(-2M) \simeq \O_Y(A)$ for a smooth rational normal sextic $A$ spanning $\PP^6$. 
Let $C \in |H|$ be general. Since $\I_{Y/T}(2) \cong \O_T$ and $\I_{C/Y}(2) \cong \O_Y(A)$, the exact sequence of
ideal sheaves becomes 
$$
0 \longrightarrow \O_T \longrightarrow \I_{C/T}(2) \longrightarrow \O_Y(A) \longrightarrow  0,
$$
which shows that $C$ is contained in a pencil of surfaces $P_{C+A} :=  \vert \I_{C/T}(2) \vert$ with base scheme $C+A$.  Actually $P_{C+A}$  is a general line in   $P_A := \vert \mathcal I_{A/T}(2) \vert$. 
 Let us recall from \cite{Ve} some properties of $P_A$. We have $\dim P_A = 9$ and any smooth $Y \in P_A$ is a Nikulin surface. Let $N_Y$ be the sum of the $8$ lines of $Y$ defined by its Nikulin lattice. Notice that these are bisecant to $A$ and that $\mathcal O_Y(N_Y) \cong \mathcal O_Y(C-A)$. Moreover, the union of the bisecant lines to $A$ contained in $T$ is a singular element $Y_o \in P_A$, whose normalization $\nu: R \to Y_o$ is a $\mathbb P^1$-bundle $p: R \to \mathbb P^1$. It turns out that
$$
\nu^* P_A\vert_{Y_0} = p^*\vert \mathcal O_{\mathbb P^1}(8) \vert + \nu^*A.
$$
In particular, $\nu^*N_Y \in p^*\vert \mathcal O_{\mathbb P^1}(8) \vert$. Relying on these properties, we claim that the moduli map $m_A: P_A \dashrightarrow \F_8^{\n,s}$ is not constant. To prove our claim we observe that $m_A(Y) = m_A(Y')$ if and only if there exists $\alpha \in \Aut \mathbb P^6$ so that $\alpha(Y) = Y'$ and $\alpha^*\O_{Y'}(C) \cong \O_Y(C)$.
Since $A + N_Y \sim C$ and $A$ is rigid, it follows that the latter condition is equivalent to $\alpha^*(N_{Y'}) = N_Y$. Now consider the moduli map
$n_A: P_A \dashrightarrow \vert \mathcal O_{\mathbb P^1}(8) \vert / PGL(2)$ sending $Y$ to the $PGL(2)$-orbit of $ p_*N_Y$. The latter condition clearly implies that $n_A$ factors through $m_A$. Since
$n_A$ is not constant, the same is true for $m_A$.  This implies that $m_A$ is not constant on a general  pencil $P_{A+C}$ and the statement follows. 
\end{proof}

\section{Half $K3$ surfaces and half Nikulin surfaces} \label{sec:halfnikulin}

\subsection{Limits of $K3$ surfaces} \label{ss:limits}

By results of Kulikov \cite{Ku} and Persson-Pinkham \cite{pp},
  semistable degenerations of $K3$ surfaces are completely classified and of three types. In the type II case, the central fiber is a chain of elliptic ruled surfaces with a rational component at each end, and all double curves are smooth elliptic curves. Furthermore, all elliptic components can be contracted performing suitable birational modifications and thus leaving only the two rational surfaces. We therefore use the following terminology.

\begin{definition}\label{tII}
  A \emph {$K3$ surface of type II} is the transversal union $X \sqcup_A X'$ of two smooth rational surfaces $X$ and $X'$ glued along a smooth elliptic curve $A$ that is anticanonical on both surfaces.
It is \emph{stable} (\cite[(3.1)]{fri}) if in addition $\N_{A/X} \* \N_{A/X'} \cong \O_A$. 
\end{definition}

If $p:\X\to\mathbb{D}$ is a proper flat map from a threefold $\X$ to a disc $\mathbb{D}$ whose general fibers are smooth irreducible $K3$ surfaces and whose central fiber is a type II $K3$ surface $S$, then $S$ is said to be {\em smoothable}. If moreover the total family $\X$ is smooth, then $p$ is called a {\em semi-stable degeneration}. 

We recall that, if $S = X \sqcup_A X'$ is any transversal union of two smooth surfaces along a smooth curve $A$, then the {\it first cotangent sheaf of $S$}, defined as 
$T^1_{S}:=\ext^1_{\O_S}(\Omega_{S},\O_{S})$,  cf. \cite[Cor. 1.1.11]{ser} or \cite[\S 2]{fri}, satisfies 
\[
  T^1_{S} \cong \N_{A/X} \* \N_{A/X'},
\]
by \cite[Prop. 2.3]{fri}. Thus, the second part of Definition \ref{tII} can be rephrased by saying that a type II $K3$ surfaces is stable if and only if its first cotangent sheaf is trivial. We refer to 
\cite[Chp. 2]{ser} or \cite[\S 2]{fri} for the deformation-theoretic meaning of this sheaf. 

A crucial point is that any {\it stable} $K3$ surface of type II occurs as the central fiber of a semi-stable degeneration of $K3$ surfaces by \cite[Prop. 2.5, Thm. 5.10]{fri}. Moreover, any $K3$ surface $S = X \sqcup_A X'$ of type II with $h^0(\N_{A/X} \* \N_{A/X'}) >0$ can be made stable after a suitable birational modification as follows.  If  $T^1_S$ is nontrivial, pick any element $Z$ in the linear system $|T^1_S|=|\N_{A/X} \* \N_{A/X'}|$ on $A$; also choose any ``decomposition'' $Z=W+W'$ into effective divisors  on $A$. Then $W$ (respectively, $W'$) is a $0$-dimensional subscheme 
of $X$ (resp., $X'$). Let $\widetilde{X} \to X$ (resp.,
$\widetilde{X}' \to X'$)  be the blow up along $W$ (resp., $W'$) and denote by $\widetilde{A}$ the strict transform of $A$ on both surfaces. Let 
$\widetilde{S}:=\widetilde{X}  \sqcup_{\widetilde{A}} \widetilde{X}'$ denote the natural gluing along $\widetilde{A}$. Then $\widetilde{S}$ is a stable $K3$ surface of type II. We will refer to the natural map
\begin{equation} \label{eq:birmod}
\pi: \widetilde{S}=\widetilde{X}  \sqcup_{\widetilde{A}} \widetilde{X}' \longrightarrow X \sqcup_A X'=S 
\end{equation}
as {\it a birational modification along $Z \in |T^1_S|$}. Note that this is not unique, as it depends on the choice of the decomposition $Z=W+W'$.

If a $K3$ surface $S = X \sqcup_A X'$ of type II is smoothable, then $\pi$ can be achieved by performing birational modifications on the whole threefold $\X$. Indeed, the family $p: \X \to \mathbb{D}$ determines a nonzero section of $T^1_S$ and the threefold $\X$ is singular precisely along  the zero set $Z \in |T^1_S|$
of this section, cf., e.g.,  
\cite[Chp. 2]{ser} or \cite[\S 2]{fri}. The singularities can be resolved by a small resolution whose restriction to the central fiber yields \eqref{eq:birmod}. The resolution and its nonuniqueness can be easily explained when $Z$ consists  of distinct  points: the tangent cone to $\X$ at each of these points has rank $4$. The exceptional divisors of the blow up $\widehat{\X} \to \X$ at these points are rank $4$ quadric surfaces. These can be contracted along any of the two rulings on one of the two irreducible components of the strict transform of $S$ in $\widehat{\X}$ by a contraction map $\widehat{\X} \to \widetilde{\X}$. One obtains a morphism $\widetilde{\X} \to \X$, which is the desired small resolution, and depends on the choice of ``which of the components of the central fiber to contract along'', corresponding to the choice of decomposition $Z=W+W'$ above. If $Z$ is nonreduced, the situation can be handled in a similar, only technically more involved, way.

\begin{remark} \label{rem:degpos}
  The above discussion yields in particular that the effectiveness of the line bundle $\N_{A/X} \* \N_{A/X'}$ on $A$ is a necessary condition  for a  type II $K3$ surface $S= X \sqcup_A X'$ to be \emph{smoothable}. Conversely, up to a birational modification that can be fixed to be the identity on either component, this condition is also sufficient.
\end{remark}

For every integer $g \geq 2$, let $\F_{g}$ be the coarse moduli space parametrizing smooth irreducible primitively polarized $K3$ surfaces $(Y,H)$ of genus $g$, that is, $Y$ is a smooth $K3$ surface and $H\in\Pic(Y)$ is a big and nef line bundle that is indivisible  in $\Pic Y$ and satisfies $H^2=2g-2 \geq 2$. 

A polarization $H$ on a $K3$ surface $S$ of type II is still defined as a big and nef line bundle on $S$. If $S$ is stable, then it naturally carries a Cartier divisor
\begin{equation} \label{eq:canxi}
 \xi \in \Pic S \; \; \mbox{satisfying} \; \; \xi|_{X}\simeq \O_{X}(A) \; \; \mbox{and} \; \; \xi|_{X'}\simeq\O_{X'}(-A)
\end{equation}
(cf. \cite[(3.3)]{fri2}). Two polarizations $H_1$ and $H_2$ on $S$ are  called  {\it equivalent} if $H_1 \cong H_2 \* \xi$.  A polarization is {\it primitive} if its image in $H^2(S,\ZZ)/\langle \xi \rangle$ is indivisible, cf. \cite[(3.11)]{fri2}.

By \cite[Thm. 4.10]{fri2} there is a partial compactification
$\overline{\F}_g$ of $\F_g$ whose boundary consists of divisorial components parametrizing various kinds of type II stable degenerations of $K3$ surfaces. More precisely, the points of $\overline{\F}_g \setminus \F_g$ represent isomorphism classes of triples 
\[ \left(S:=X \sqcup_A X', Z, H\right) \]
where $S$ is a $K3$ surface of type II, $Z$ is an element in $|T^1_S|$ and $H$ is an equivalence class of primitive polarizations on $S$.  Restricting to a dense open subset of 
$\overline{\F}_g$, one may work with a specific component in the boundary and thus assume that the divisor  $\overline{\F}_g \setminus \F_g$ is irreducible (as we will often do without further notice). One of the main achievements of this paper is to describe induced partial compactifications of the loci $\F_g^{\n,s}$ (respectively, $\F_g^{\n,ns}$) in $\F_g$ parametrizing Nikulin surfaces of standard (resp., non-standard) type, 
and of the loci $\widetilde{\F}_{2g-1}^{\n,s}$ (respectively, $\widetilde{\F}_{2g-1}^{\n,ns}$) in $\F_{2g-1}$ parametrizing surfaces with a Nikulin involution of standard (resp., non-standard ) type, cf., e.g., Corollaries \ref{cor:exparcomI} and \ref{cor:exparcomsopraI}. 

In the sequel we will make use of the following result.

\begin{lemma} \label{lemma:utilissimo}

 Let $S= X \sqcup_A X'$ be the transversal union of two irreducible projective surfaces $X$ and $X'$ along a smooth irreducible curve $A$ lying in the smooth locus of both $X$ and $X'$. Let $C \subset S$ be a nodal curve, which is Cartier on $S$ and smooth outside of $A$ (in particular, $C$ is disjoint from $\Sing X \cup \Sing X'$). 

Assume that $X$ admits a deformation to an irreducible surface that deforms $C$ preserving a subset $Z$ of its nodes; more precisely, there is:
\begin{itemize}
\item[(i)]  a flat proper map $p: \X \to \mathbb{D}$ over the disc $\DD$ whose general fibers are irreducible and with central fiber $p^{-1}(0)\simeq S$,
\item[(ii)]  a relative Cartier divisor $\C \subset \X$ 
such that $p|_{\C}^{-1}(0)\simeq C$,
\item[(iii)]  a one-dimensional subscheme $\Z \subset \C$ such that $p|_{\Z}^{-1}(0)=Z\subset C$,
\item[(iv)] for all $t \neq 0$, the fiber $\C_t:= p|_{\C}^{-1}(t)$ is nodal and its scheme of nodes coincides with  $\Z_t:= p|_{\Z}^{-1}(t)$.
\end{itemize}

Then there is a nonzero section $\sigma \in H^0(\N_{A/X} \* \N_{A/X'})$ such that
$Z$ is contained in the zero scheme $Z(\sigma)$ of $\sigma$.
\end{lemma}

\begin{proof}
Since $C$ does not meet $\Sing X$ and $\Sing X'$, we may assume both $X$ and $X'$ to be smooth.

If the total space $\X$ is singular along the double curve $A$, we blow up $\X$ along $A$, and repeat the process if necessary as in \cite[\S 2(b-c)]{gh}, until we get a morphism $f:\widetilde{\X} \to \X$, where $\widetilde{\X}$ has isolated singularieties. The result is a new deformation $\tilde{p}: \widetilde{\X} \to \mathbb{D}$ with unchanged general fibers and  new central fiber 
\[ \widetilde{S}:={\widetilde p}^{-1}(0)=X_0 \cup X_1 \cup \cdots \cup X_r \cup X_{r+1}, \; \;  X_0=X, \; X_{r+1}=X',\]
consisting of a chain of ruled surfaces $X_1,\ldots,X_r$ over $A$ (as in the picture of \cite[p.~38]{gh}) with $X$ and $X'$ attached at its ends and such that for each $i=1,\ldots,r$ the intermediate surface $X_i$ has two sections $A_i$ and $A_{i+1}$ coinciding with the intersections $X_i\cap X_{i-1}$ and $X_i\cap X_{i+1}$, respectively. Here we are identifying $A_1$ with $A$ on $X_0= X$ and $A_{r+1}$ with $A$ on $X_{r+1}= X'$. Note that $X_i = \PP(\E_i)$ with $\E_i$ a rank two  bundle on $A$ such that $\E_i \simeq\O_A \+ \L$ and  $\deg \L=0$. If $\L =\O_A$, then $X_i \cong A \x \PP^1$. Otherwise, $X_i$ has two natural sections, corresponding to the  normalized bundles $ \O_A \+ \L$ and $ \O_A \+ \L^{-1}$, whence the  sections have normal bundles $\L$ and $\L^{-1}$, respectively. In any case one has
\begin{equation}
  \label{eq:degnormali}
  \deg \N_{A_i/X_i} = \deg \N_{A_{i+1}/X_i} =0, \; \; i=1,\ldots, r 
\end{equation}
and
\begin{equation}
  \label{eq:normali}
  \N_{A_i/X_i} \* \N_{A_{i+1}/X_i} \cong \O_A, \; \; i=1,\ldots, r. 
\end{equation}

Write $C = D \cup D'$ with $D \subset X$ and $D' \subset X'$ its smooth irreducible components. The intersection $D \cap D'$ is transversal and occurs along $A$.
The central fiber of the new relative Cartier divisor $\widetilde{\C}=f^*(\C)$
is 
\[ \widetilde{C}= D_0 \cup D_1 \cup \cdots \cup D_r \cup D_{r+1}, \; \;  D_0=D, \; D_{r+1}=D',\]
where for every $i=1,\ldots,r$ the curve $D_i$ is contained in $X_i$ and consists of disjoint lines of its ruling. Hence, $D_1 \cup \cdots \cup D_r$ is a union of chains of smooth rational curves, each chain connecting the pair of points on $D$ and $D'$ mapping to a node of $C$. In particular, $\widetilde{C}$ has only nodes as singularieties and they all lie along the double curves of $\widetilde S$.

The deformation $\tilde{p}: \widetilde{\X} \to \mathbb{D}$ determines an element $\xi \in \Ext^1_{\O_{\widetilde{S}}}(\Omega_{\widetilde{S}},\O_{\widetilde{S}})$, namely, the {\it Kodaira-Spencer class}. We have the local-to-global exact sequence for $Ext$:  
\[
\xymatrix{ 0  \ar[r] & H^1(\hom_{\O_{\widetilde{S}}}(\Omega_{\widetilde{S}},\O_{\widetilde{S}})) \ar[r] &  \Ext^1_{\O_{\widetilde{S}}}(\Omega_{\widetilde{S}},\O_{\widetilde{S}}) \ar[r]^{\hspace{0.7cm}\alpha} & H^0(T^1_{\widetilde{S}}),} \] 
where 
\begin{equation} \label{eq:friolus}
 T^1_{\widetilde{S}}:=\ext^1_{\O_{\widetilde{S}}}(\Omega_{\widetilde{S}},\O_{\widetilde{S}}) \cong \+_{i=1}^{r+1} \left(\N_{A_i/X_{i-1}} \* \N_{A_i/X_i}\right) 
\end{equation}
by \cite[Prop. 2.3]{fri}. Let $\tilde{\sigma}:=\alpha(\xi)$. By \cite[Rem. 2.6]{fri}, the singularities of $\widetilde{\X}$ along $\widetilde{S}$ coincide with the zero set 
 $Z(\tilde{\sigma})$. As $\widetilde{\X}$ has only isolated singularieties, 
\eqref{eq:degnormali} yields
\begin{equation}
  \label{eq:normtriviali}
  \N_{A_i/X_{i-1}} \* \N_{A_i/X_i} \cong \O_A \; \; i=2,\ldots, r,
\end{equation}
whence \eqref{eq:friolus} yields that
$Z(\tilde{\sigma}) =Z(\tilde{\sigma}_1) \sqcup Z(\tilde{\sigma}_{r+1})$,
with  
\[ \tilde{\sigma}_1 \in H^0(\N_{A_1/X_0} \* \N_{A_1/X_1}), \; \; \mbox{and} \; \; \tilde{\sigma}_{r+1} \in H^0(\N_{A_{r+1}/X_r} \* \N_{A_{r+1}/X_{r+1}}).
\]
 An easy local computation, using the fact that $\widetilde{\C}$ is Cartier, shows that a node of $\widetilde{C}$ lying outside of $\Sing \widetilde{\X}$ automatically smooths as $\widetilde{S}$ deforms, see for instance \cite{Ch}, \cite[\S 2]{gal} or \cite[Pf. of Lemma 3.4]{galknu}. Hence, the preserved set of nodes of $\widetilde{C}$ lie in $Z(\tilde{\sigma})$. Pushing forward via $f$, the preserved set of nodes $Z$ of $C$ must lie in
\[|\left(\N_{A_1/X_0} \* \N_{A_1/X_1}\right) \* \left(\N_{A_{r+1}/X_r} \* \N_{A_{r+1}/X_{r+1}}\right)| 
 =  |\N_{A/X} \* \N_{A/X'}|,\]
since $\N_{A_1/X_1} \* \N_{A_{r+1}/X_{r}} \cong \O_A$ by \eqref{eq:normali} and \eqref{eq:normtriviali}.
\end{proof}

\subsection{Half $K3$ surfaces and half Nikulin surfaces} \label{ss:defhalf}

The scope of this section is to study  properties of surfaces arising as one of the two components of a type II $K3$ surface.

\begin{definition}
A smooth rational surface $X$ is called a {\em half $K3$ surface}  if it carries a smooth irreducible anticanonical divisor. The \emph{degree of $X$} is $K_X^2$.
\end{definition}

\begin{remark}\label{rem:totti}
Although we will not use this, we note that a half $K3$ surface $X$ of degree $d$ occurs as component of a stable type II $K3$ surface if and only if 
$d \geq -9$. 

The ``only if'' part directly follows from Remark \ref{rem:degpos} and the fact that any half $K3$ has degree $\leq 9$.
 For the ``if''part, fix any smooth elliptic anticanonical curve $A$ on $X$ and 
embed it into $\PP^2$ as a cubic; in the special case where $d=-9$, choose a particular embedding given by
any triple root of the inverse of its normal bundle $\N_{A/X}$, that is, a line bundle $\L$ of degree $3$ on $A$ such that  $\L^{\*3} \cong \N_{A/X}^{\vee}$. 
Since $\N_{A/X} \* \N_{A/\PP^2}$ is effective, the surface $S:=X \sqcup_A \PP^2$ is a smoothable type II $K3$ surface by Remark \ref{rem:degpos}.
\end{remark} 

We are interested in the natural candidates among half $K3$ surfaces to be irreducible components of limits of Nikulin surfaces.

\begin{definition}
  A half $K3$ surface $X$ is called an \emph{untwisted} (respectively, \emph{twisted}) \emph{half Nikulin surface} if it contains four disjoint smooth rational curves $N_1,\ldots, N_4$ with $N_i^2=-2$ such that
$N:=N_1 + \cdots + N_4$ (resp., $N+K_X:=N_1 + \cdots +N_4 +K_X$) is $2$-divisible in $\Pic X$. 
 A  \emph{ (primitively) polarized half Nikulin surface} is a pair $(X,H)$ with $X$ half Nikulin and $H$ a big and nef (primitive) line bundle on $X$ such that $H \cdot N=0$. 
\end{definition}

Note that  on a half Nikulin surface $X$ the four ($-2$)-curves $N_i$ are always disjoint from any irreducible anticanonical divisor. Moreover, in the untwisted case the $2$-divisible line bundle on $X$ uniquely defines a finite double covering
$
 \pi: \widehat X \longrightarrow X
 $
 branched along $N$. Similarly, if $X$ is twisted, for any fixed smooth and irreducible anticanonical divisor $A$, there is a finite double cover (still denoted by $\pi$) branched along $N+A$.  Since $X$ is smooth, $\widehat X$ is smooth as well. For $i=1,\ldots,4$, we set $ \widehat N_i := \pi^{-1} (N_i)$;  since $\pi^*N_i = 2\widehat N_i$, it follows that $\widehat N_i$ is a $(-1)$-curve. 

We denote by $\widehat \tau$ the contraction of $\widehat N_1, \ldots, \widehat N_4$; the surface $\widetilde{X}:=\widehat{\tau}(\widehat{X})$ is still smooth.
One has the Cartesian square
\begin{equation}\label{cover}
 \xymatrix{
 {\widehat X}  \ar[r]^{\hat \tau} \ar[d]_{\pi}&   {\widetilde X} \ar[d]^{\bar{\pi}} \\
 X  \ar[r]^{\tau} & {\overline X},}
 \end{equation}
 where $\tau$ is the contraction of the curves $N_i$ to four nodes on $\overline{X}$ and $\bar{\pi}$ is the quasi-\'etale double cover branched
 on $\Sing \overline X$ and, in the twisted case, on  $\tau(A) \cong A$. It is also clear that $\bar{\pi}$ is the quotient map by the involution
$\iota: \widetilde X \longrightarrow \widetilde X$
induced by $\pi$. As a consequence, in the twisted case the restriction $\overline{\pi}|_{\widetilde{A}}: \widetilde{A} \longrightarrow \tau(A)\cong A$
is  an isomorphism and $ K_{\widetilde X} \sim \bar{\pi}^* K_{\overline X}+\widetilde{A} \sim -  \bar{\pi}^*(A)+\widetilde{A}\sim \widetilde{A}$,
so that $\widetilde{A}$ is anticanonical.  Instead, in the unwisted case the canonical divisor of $\widetilde X$ satisfies $ K_{\widetilde X} \sim \bar{\pi}^* K_{\overline X}$ and, if $A$ is any smooth anticanonical curve on $X$ and $\widetilde A:=\bar{\pi}^{-1}(A)$, the map $\bar{\pi}|_{\widetilde{A}}: \widetilde{A} \longrightarrow \tau(A)\cong A$ is an {\'e}tale double cover; in particular, $\widetilde A$ is again a smooth anticanonical  divisor of $\widetilde X$. 

To summarize, we have:

\begin{lemma} \label{2d}
  The surface $\widetilde X$ is a half $K3$ surface of degree $2d$ in the untwisted case and $d/2$ in the twisted case.
\end{lemma}

\begin{remark}
 Let $X$ be a 
half Nikulin surface of degree $d$ glued with another half Nikulin surface $X'$ to obtain a type II $K3$ surface that is smoothable to Nikulin surfaces. Then, Lemma \ref{2d} and Remark \ref{rem:totti} yield $d\geq -4$. 

Indeed, this is immediate in the untwisted case. In the twisted case, the $2$-divisibility of $N+K_X$ forces $d$ to be divisible by $4$, so we only need to eliminate the case $d=-8$. Assume that $S=X \sqcup_A X'$ is the flat limit of smooth Nikulin surfaces. After making a birational modification of $S$ leaving $X$ fixed, we may assume that $X \sqcup_A X'$ is stable,  which means that $d'=-d=8$. But the only half $K3$ surfaces of degree $8$ are $\PP^1 \x \PP^1$, $\FF_1$ and $\FF_2$, and neither contain four $(-2)$-curves, a contradiction.
\end{remark}

\begin{definition} \label{not:cover}
  We call \eqref{cover} the {\em double cover diagram associated with $X$} (which depends on the choice of an anticanonical curve $A$ in the twisted case), the surface $\overline{X}$ the {\em nodal model of $X$} and the surface $\widetilde{X}$ the {\em half $K3$ double cover of $X$}. 
\end{definition}

\subsection{Reconstructing half $K3$ surfaces from hyperelliptic hyperplane sections} \label{ss:recon}

A crucial point in the proof of our results is to  reconstruct,
up to finitely many choices, a polarized half $K3$ surface  from its general hyperplane section. This is troublesome in general but becomes easier if the hyperplane section is hyperelliptic. Indeed, under mild conditions, the $g^1_2$ on the curve turns out to be induced by a unique pencil of divisors on the surface, which can be exploited in the reconstruction process. 

\begin{definition} Let $X$ be a half $K3$ surface. A \emph{pencil of conics} $\vert B \vert$ on $  X$ is
a base point free pencil of divisors $B$ such that $B^2 = 0$ and $B\cdot K_{  X} = -2$. 
\end{definition}

\begin{lemma} \label{lemma:reider}
Let $X$ be a half $K3$ surface.
Assume that $D \subset X$ is a smooth, hyperelliptic curve of genus $g(D) \geq 2$ such that $D^2 \geq \max\{10,2g(D)+1\}$. 
 Then there is a unique pencil of conics $|B|$ on $X$ cutting out the $g^1_2$ on $D$ (in particular, $B \cdot D=2$).
\end{lemma}

\begin{proof}
We fix a reduced and irreducible member $A \in |-K_X|$. From the natural restriction sequence 
 and the hypotheses on $X$, we find that $H^0(K_X+D)\cong H^0(\omega_D) \cong \CC^{g}$. In particular, $K_X+D$ is effective and nontrivial and the linear system $|K_X+D|$ fails to separate any pair of points $x$ and $y$ forming a divisor in the $g^1_2$ on $D$. Since $D^2\geq 10$, we can apply Reider's Theorem and conclude that any such pair is contained in an effective divisor $B$ on $X$ satisfying $(B \cdot D,B^2) \in \{(0,-2),(0,-1),(1,-1),(1,0),(2,0)\}$. As $B \cap D$ contains both $x$ and $y$, the only possibility is $(B\cdot D,B^2)=(2,0)$. Since the divisor $x+y$ moves in the hyperelliptic pencil, one has $h^0(B) \geq 2$. Let $|B|=|M|+F$ be the decomposition of $B$ into its moving  and fixed parts. As members of $|B|$ must pass through varying members of the $g^1_2$, we must have $M \cdot D=2$ and $F \cdot D=0$. The Hodge Index Theorem thus yields $M^2=0$, 
and, since $D$ is not rational, one concludes that $h^0(M)=2$. 
Hence, without loss of generality, we may assume that $|B|$ is a base point free pencil cutting out the $g^1_2$ on $D$ and satisfying $B^2=0$. 
 
The adjunction formula yields
$B \cdot A =-B \cdot K_X = -B \cdot (B+K_X) =2-2g(B) \leq 2$. As $B$ is nef, we must have $B \cdot A =0$ or $2$. If $B \cdot A=0$, then $A$ is contained in a member of $|B|$, whence 
$-D \cdot K_X =D \cdot A \leq D \cdot B=2$, so that 
$D^2=2g-2-D \cdot K_X \leq 2g$, a contradiction. Hence $B \cdot A =-B \cdot K_X=2$, as claimed.

We have left to prove the unicity of $B$. If $|B'|$ is a different pencil satisfying the same conditions, then $B' \cdot B \geq 2$ as the members of the pencils pass through the same pairs of points on $D$. Thus, $(B'+B)^2=2B' \cdot B \geq 4$, and
the Hodge Index Theorem yields  
$ 4D^2 \leq (B + B')^2D^2 \leq ((B + B')\cdot D)^2 =16$,
whence the contradiction $D^2 \leq 4$.
\end{proof}
 \begin{remark}
  On Del Pezzo surfaces of degrees $\geq 3$, the existence of $B$ (under weaker assumptions) is an immediate consequence of \cite[Cor. 5.2]{Kn}. On any Del Pezzo surface, the result  can be deduced from \cite[Prop. 5.1]{Kn} with a similar reasoning. However, in the present paper, Lemma \ref{lemma:reider} is applied also to surfaces that are not Del Pezzo.
\end{remark}

\begin{proposition} \label{prop:recon}
  Fix an integer $b \leq 9$ and set $s:=\max\{0,\lfloor \frac{-b+1}{2}\rfloor\}$. Let $D$ be a smooth hyperelliptic curve of genus $g(D) \geq 2$ carrying a line bundle $\N$ of degree 
\begin{equation} \label{eq:condcazzo}
a > 2g(D)+6+2s.
\end{equation} 
Then there are finitely many half $K3$ surfaces $X$ of  degree $b$ such that $D$ is contained in $X$ and $\N_{D/X} \cong \N$. 

If moreover 
\begin{equation}
  \label{eq:-Kampiuccio}
  -K_X \cdot \gamma >0 \; \mbox{for all curves} \; \gamma \; \mbox{for which} \; D \cdot \gamma >0
\end{equation}
(e.g., $-K_X$ is ample), then condition \eqref{eq:condcazzo} can be weakened allowing equality (and $s=0$).
\end{proposition}

\begin{proof}
  Let $X$ be any such surface. By Lemma \ref{lemma:reider}, there is a unique pencil of conics $|B|$ cutting out the $g^1_2$ on $D$. Let $p: X \to \PP^1$ be the fibration induced by $|B|$, thus the restriction $q:=p|_{D}: D \to \PP^1$ is the hyperelliptic map. The line bundle $\N$ is very ample for degree reasons. The standard 
exact sequence
\begin{equation}\label{standard}
0 \longrightarrow \mathcal O_{X} \longrightarrow \mathcal O_{X}(D) \longrightarrow \mathcal O_{D}(D) \cong \N \longrightarrow 0
\end{equation}
shows that $\O_X(D)$ is globally generated and thus defines a morphism
\[ \varphi: X \rightarrow \overline{X} \subset \PP H^0(\O_{X}(D))^\vee.\]
Since $\varphi$ restricts to an isomorphism on $D$, then the general hyperplane section of $\overline{X}$ is smooth. In particular, $\varphi$ is birational and $\overline{X}$ has at most isolated singularities. 

Apply the $p_*$-functor to \eqref{standard} to obtain
\begin{equation}\label{vb}
0 \longrightarrow \mathcal O_{\PP^1} \longrightarrow p_* \mathcal O_{X}(D) \longrightarrow q_*\N \longrightarrow 0.
\end{equation}
Here, $\N_1:=p_* \mathcal O_{X}(D)$ is a globally generated rank-$3$ vector bundle on $\PP^1$ of  degree $a$. Hence, there are only finitely many choices for it once $a$ is given.

We set $P := \PP(\N_1^\vee)$ and $R := \PP (q_*\N^\vee)$.
Then $P$ and $R$ are, respectively, a $\PP^2$-bundle and a $\PP^1$-bundle over $\PP^1$ and, by \eqref{vb}, $R$ is embedded in $P$ as a section of $\O_{P}(1)$.  Moreover, there is a natural morphism $\varphi^\dagger: X \to P$, whose image we denote by $X^\dagger$, through which $\varphi$ factors, and $X^\dagger$ is embedded in $P$ as a conic bundle. Denoting by $F$ a fiber of $P \to \PP^1$, standard computations show that
\begin{equation}\label{sta}
X^\dagger \sim 2R - (a-2g(D)-2)F.
\end{equation}
 The curve $D$ is embedded in $R\subset P$ so that $R \to \PP^1$ restricts to the hyperelliptic map $q$.
It is quite immediate that $D = R \cap X^\dagger$. To end the proof, it is enough to show that 
\begin{equation}\label{goal}
h^0(\mathcal I_{D / P}(X^\dagger)) = 1.
\end{equation}

We consider the following standard exact sequence of ideal sheaves on $P$:
$$
0 \longrightarrow \mathcal O_P(-X^\dagger - R) \longrightarrow \mathcal O_P(-X^\dagger) \oplus \mathcal O_P(-R) \longrightarrow \mathcal I_{D/ P} \longrightarrow 0.
$$
Tensoring it by $\mathcal O_P(X^\dagger)$ and taking cohomology, one sees that \eqref{goal}
is equivalent to $h^0(\O_P(X^\dagger-R))=0$. It is immediate to check that this is also equivalent to the vanishing $h^0(\O_{X^\dagger} (X^\dagger - R))=0$. If this failed, the line bundle $(\varphi^\dagger)^*\O_{X^\dagger} (X^\dagger - R)$ on $X$ would have sections, that is, by \eqref{sta},  the divisor
\begin{equation*}
\gamma:=D-(a-2g(D)-2) B
\end{equation*}
would be effective (and  nontrivial, as $\gamma \cdot B=2$). The definition of $s$ ensures that $-K_X+sB$ is nef (as it intersects the irreducible effective divisors $-K_X$ and $B$ nonnegatively), and \eqref{eq:condcazzo} yields that
$\gamma \cdot(-K_X+sB)<0$, whence $\gamma$ is not effective, and we are done.

Finally, assume that equality in \eqref{eq:condcazzo} and condition \eqref{eq:-Kampiuccio} hold. The equality $D^2=a=2g(D)+6+2s$ yields
$-K_X \cdot D >0$, and thus $-K_X^2=b > 0$ by \eqref{eq:-Kampiuccio} and $s=0$. As a consequence, one has $\gamma=D-4B$ and $-\gamma \cdot K_X=0$. One computes $\gamma \cdot D =a-8=2g(D)-2 >0$, whence $\gamma$ is not effective by \eqref{eq:-Kampiuccio}.
\end{proof}

\section{Construction of untwisted half Nikulin surfaces} \label{sec:const-unt}

\subsection{Main series of examples} \label{ss:ex-unt} 

We start with an example of untwisted half Nikulin surfaces endowed with a pencil of conics.

\begin{example} \label{ex:4nodal}
  Consider the four-nodal cubic surface 
$$
\overline{X}:=\lbrace x_2x_3x_4 + x_1x_3x_4 + x_1x_2x_4 + x_1x_2x_3 = 0 \rbrace \subset \PP^3 
$$
and let $X$ be its desingularization. There are four disjoint smooth rational curves $N_1,\ldots,N_4$ of self-intersection $(-2)$ on $X$ that are contracted to the four nodes $n_i$ of $\overline{X}$ by the anticanonical
morphism $f:   X \to \PP^3$. The four-nodal cubic $\overline{X}$ contains the edges $T_{ij} := \langle n_i n_j \rangle$ of the tetrahedron $T$
of vertices $n_1, \ldots, n_4$. Let $E_{ij}$ be the strict transform of $T_{ij}$ by $f$. It is clear that $E_{ij}^2 = -1$ and $N_i\cdot E_{ij} = N_j \cdot E_{ij}= 1$. Moreover it is a standard exercise to show that
 \begin{itemize}
\item  $\vert N_1 + N_2 + 2E_{12} \vert =  \vert N_3 + N_4 + 2E_{34} \vert$,
\item  $ \vert N_1 + N_3 + 2E_{13} \vert = \vert N_2 + N_4 + 2E_{24} \vert $,
\item   $\vert N_1 + N_2 + 2E_{12} \vert = \vert N_3 + N_4 + 2E_{34} \vert$
\end{itemize}
and that these three linear systems are three pencils of conics such that the listed elements are mapped to double lines in $\PP^3$ by $f$. Two elements of any of these pencils are linearly equivalent to $N_1 + N_2+N_3+N_4$ plus twice the sum of two of the $E_{ij}$, showing that $N_1 + N_2+N_3+N_4$ is $2$-divisible in $\Pic X$.
\end{example}

We now perform a different construction. Fix a quadrilateral of lines
$$
N:= N_1 + N_2 + N_3 + N_4
$$
on $\PP^1 \times \PP^1$,
where $N_1, N_2$ and $N_3 , N_4$ respectively are in $\vert \mathcal O_{\PP^1 \times \PP^1}(1,0) \vert$ and $\vert \mathcal O_{\PP^1 \times \PP^1}(0, 1) \vert$, with $N_1 \neq N_2$ and $N_3 \neq N_4$. Therefore, $\Sing N$ consists of the points $$ e_{ij} := N_i \cap N_j $$ with $1 \leq i \leq 2$ and $3 \leq j \leq 4$. Let 
$$
\sigma:   X_4 \longrightarrow \PP^1 \times \PP^1
$$
be the blowing up of these four points, 
so that $X_4$ is a half $K3$ surface, more precisely, a weak Del Pezzo surface of degree $4$. 

    \begin{remark} \label{rem:startA}
The construction of $X_4$ is equivalent to choosing a smooth elliptic curve $A$
 with two different line bundles $L_1,\,L_2\in\mathrm{Pic}^2(A)$ such that
$L_1^{\*2} \cong L_2^{\*2}$, and choosing a general point $x \in A$. Indeed, $L_1$ and $L_2$ provide an embedding $A\subset\mathbb{P}^1\times \mathbb{P}^1$ and the fact that $L_1^{\*2} \cong L_2^{\*2}$ ensures that $x$  uniquely determines a quadrilateral of lines $N_1+N_2+N_3+N_4$ in $\mathbb{P}^1\times \mathbb{P}^1$ such that the four intersection points $e_{ij}:=N_i\cap N_j$ for $1\leq i\leq 2$ and $3\leq j\leq 4$ all lie on $A$ and $e_{14}=x$. We recall, for later use, the following relations on $A$:
\begin{equation}\label{doppi}
2e_{13}\sim 2e_{24}\,\,\textrm{ and }\,\,2e_{14}\sim2e_{23}
\end{equation}
\end{remark} 
\vspace{0.2cm}

We let $E_{ij} := \sigma^{-1}(e_{ij})$ denote the four exceptional curves on $X_4$. Moreover, by abuse of notation, we will still denote the strict transform of $N_j$ on $X_4$ by $N_j$. 
These are four disjoint smooth, rational curves of self-intersection $-2$.

The strict transforms of the pencils of conic sections through $e_{13}, e_{24}$ and $e_{14}, e_{23}$, respectively, define two pencils of conics on $X_4$, namely:
\begin{eqnarray}
\label{eq:pencil2} N_1 + N_4 + 2E_{14} \sim \sigma^*\O_{\PP^1 \x \PP^1}(1,1) - E_{13} - E_{24} \sim N_2 + N_3 + 2E_{23},\\
\label{eq:pencil1} N_1 + N_3 + 2E_{13} \sim \sigma^*\O_{\PP^1 \x \PP^1}(1,1) - E_{14} - E_{23} \sim N_2 + N_4 + 2E_{24}.
\end{eqnarray}

    \begin{remark}The anticanonical morphism $f:   X_4 \to \PP^4$ is well known: the map $f$ factors through the contraction $\tau:   X_4 \to \overline X_4$ of $N$ followed by an embedding 
$$
 \overline X_4  = Q_1 \cap Q_2 \subset \PP^4,
 $$
as a four-nodal complete intersection  of two rank-three quadrics $Q_1, Q_2$ such that 
$$ \Sing \overline X_4 = \lbrace f(N_1), \ldots, f(N_4) \rbrace = (\Sing Q_1 \cap Q_2) \cup (Q_1 \cap \Sing Q_2). $$ 
Each of the two pencils of conics \eqref{eq:pencil2} and \eqref{eq:pencil1} contains two double lines: $E_{14}$ and $E_{23}$ in the former case and $E_{13}$ and $E_{24}$ in the latter.
\end{remark} 

It is easy to construct from $ X_4$ examples of half $K3$  surfaces $X_d$ of any degree $d:=K_{X_d}^2 \leq 4$ endowed with a pencil of conics. Fix any smooth anticanonical divisor
$A \in |-K_{X_4}|$ (by the point of view of Remark \ref{rem:startA}, the curve $A$ is already fixed from the start) and one of the two pencils above, say \eqref{eq:pencil2}. Consider the blowing up
$$
\gamma_d:   X_d \longrightarrow   X_4
$$
at  $4 - d$ {\it sufficiently general} points $u_1,\ldots, u_{4-d} \in A$. More precisely, in order to make sure that certain linear systems only contract the four curves $N_i$ (cf. the proof of Lemma \ref{uno}), we require that
\begin{equation}
  \label{eq:condblow}
 u_1,\ldots, u_{4-d} \; \; 
\mbox{lie on distinct {\it smooth, irreducible} members of the pencil \eqref{eq:pencil2}.}
\end{equation}
\begin{remark} \label{rem:puntigenerali}
  If $d \geq 1$, the generality of the points $u_1,\ldots, u_{4-d}$ also ensures that the only curves $\gamma \subset X_d$ such that
$-K_{X_d} \cdot \gamma \leq 0$ are $N_1, \ldots, N_4$.
\end{remark}

We denote the strict transforms of $A$, $N_j$ and $E_{ij}$ on $X_d$ by the same letters, so that $A$ is a smooth anticanonical divisor on $X_d$ with $A^2=d \leq 4$. Let
\[ U_m:=\gamma_d^{-1}(u_m), \; \; 1 \leq m \leq 4-d.\]
be the exceptional divisors of $\gamma_d$, and consider the contractions
$$    X_d \stackrel {\gamma_d} \longrightarrow   X_4 \stackrel{\sigma} \longrightarrow \PP^1 \times \PP^1.$$
We define, for later use, the line bundles 
\begin{equation} \label{eq:ell}
\ell:=(\sigma \circ \gamma_d)^*\O_{\PP^1 \x \PP^1}(1,1), \; \; \ell_1:=(\sigma \circ \gamma_d)^*\O_{\PP^1 \x \PP^1}(1,0), \; \; \ell_2:=(\sigma \circ \gamma_d)^*\O_{\PP^1 \x \PP^1}(0,1)
\end{equation} 
(in the notation of Remark \ref{rem:startA}, we have ${\ell_i}|_A \cong L_i$).

The contraction of the divisor $N=N_1+N_2+N_3+N_4$ is denoted by
$$
\tau=\tau_d:   X_d \longrightarrow \overline X_d.
$$
The pullback by $\gamma_d$ of any of the previous pencils of conics 
\eqref{eq:pencil2} and \eqref{eq:pencil1} is
still  a pencil of conics whose image by $\tau$ still has two double lines.  
We denote the pullback by $\gamma_d$ of the chosen pencil \eqref{eq:pencil2} 
by $B$. Thus, we recall that 
\begin{equation}\label{conics} B \sim N_1 + N_4 + 2E_{14} \sim N_2 + N_3 + 2E_{23} 
\end{equation}
 and $|B|$ is the strict transform under $\sigma \circ \gamma_d$ of the pencil of conic sections in $ \vert \O_{\PP^1 \times \PP^1}(1, 1) \vert $ passing through the points
$e_{13}$ and $e_{24}$. Note also that we have
$$
N = N_1 + N_2 + N_3 + N_4 \sim 2B - 2E_{14} - 2E_{23},
$$
just by the definition of $B$. Obviously this implies that 

\begin{lemma} \label{lemma:div2-I}
$\mathcal O_{  X_d}(N)$ is divisible by two in $\Pic   X_d$. In particular, $X_d$ is an untwisted half Nikulin surface. 
\end{lemma} 

Precisely, with notation \eqref{eq:ell}, equivalences \eqref{conics} yield
\begin{equation} \label{n2}
\Delta:= \frac{N}{2} \sim B - E_{14} - E_{23}
\ \sim \ell-E_{13}-E_{14}-E_{23}-E_{24}.
\end{equation} 
It is easy to check that 
\begin{equation} \label{n22}
h^i(\mathcal O_{  X_d}(\pm \Delta)) = 0 \; \;  \mbox{for} \; \; i \geq 0 \; \; \mbox{and} \; \; h^0(\mathcal O_{  X_d}(2\Delta)) = 1
\end{equation} 
 and the classes $[N_1], \dots, [N_4], [\Delta]$ generate a primitive sublattice
of rank four:
$$ \mathbb M_{  X_d} \subset \Pic   X_d. $$

\subsection{Hyperelliptic polarizations} \label{ss:hyppol}

We are looking for line bundles $\mathcal O_{  X_d}(D)$ in the orthogonal lattice
$ \mathbb M_{  X_d}^{\perp} \subset \Pic   X_d $ such that $\tau_*D$  is a very ample Cartier divisor on $\overline X_d$. For our purposes the main series of examples is provided by the divisors
\begin{equation} \label{eq:defDk}
D_k:= -K_{  X_d} + kB \sim A +kB, \; \; \mbox{for a positive integer} \; \; k.
\end{equation}
It is clear that $N_i\cdot D_k = 0$ for $i = 1, \ldots, 4$, whence $D_k$ belongs to $\mathbb M^{\perp}_{  X_d}$. Moreover, 
\begin{equation}
  \label{eq:tre}
 D_k^2 = d + 4k,\;\; K_{X_d}\cdot D_k = -d - 2k, \;\; p_a(D_k) = k + 1, \;\; D_k\cdot B = 2.
\end{equation}
For later use, recalling \eqref{eq:ell}, we also note that
\begin{equation}\label{div0}
 D_k\sim 2\ell_1 +kN_1+N_3+(k+1)N_4-\sum_{i=1}^{4-d}U_i+2kE_{14},
\end{equation}
 which can be proved using the first equivalence in \eqref{conics}, and 
 \begin{eqnarray}\label{aggiunta}
 -K_{X}  \sim  2\ell-E_{13}-E_{14}-E_{23}-E_{24}- \sum_{j=1}^{4-d}U_{j}&\sim&  2\ell_1+N_{3}+N_{4} -\sum_{j=1}^{4-d}U_{j}.
\end{eqnarray}
 A peculiarity of the linear system $|D_k|$ is the following:

\begin{lemma} \label{lemma:prymhyp}
  Any smooth curve $D \in |D_k|$ is hyperelliptic, with its $g^1_2$ being cut out by $|B|$. Furthermore,  the line bundle 
$\eta_D := \mathcal O_D(\Delta)$ 
is a nontrivial $2$-torsion element of $\Pic^0 D$. More precisely, the points
$w_{14}:= E_{14}\cap D$ and $w_{23}:= E_{23}\cap D$ are Weierstrass points on $D$ such that
$
\eta_D \cong \mathcal O_D(w_{14} - w_{23}) \cong \mathcal O_D(w_{23} - w_{14}).
$
\end{lemma}

\begin{proof}
Using \eqref{n22} one may check that $\eta_D$ is nontrivial and  $\eta_D^{\otimes 2}$ has
 a global section for every $D \in \vert D_k \vert$. The assertions are thus immediate by \eqref{eq:defDk}, \eqref{conics} and \eqref{n2}.
\end{proof}

The following summarizes  other elementary properties of $D_k$ that are rather standard when $X_d$ is a {\it weak Del Pezzo surface}, that is, for $d=K_{X_d}^2 \geq 1$; even though this is the main case of application, we will also need the case $d=-4$ in \S \ref{sec:parcom2}.

\begin{lemma}\label{uno}   
Assume that $d+k \geq 1$. 

(a) One has $h^1(D_k)=h^2(D_k)=0$; in particular $\dim \vert D_k \vert = 3k + d$.

(b) If furthermore $d+2k \geq 3$, the complete linear system $|D_k|$ is base point free and defines a morphism that is an embedding except for the contraction of $N_1,\ldots,N_4$; in particular, $\tau_*D_k$ is very ample on $\overline X_d$.
\end{lemma}

\begin{proof} 
Set $L:=-K_{X_d}+D_k \sim 2A +kB$. Then $L^2=4(d+2k) \geq 8$. As $B$ is nef and by assumption $A \cdot L =-K_{X_d} \cdot L =2(d+k) \geq 2$, we have that $L$ is nef. Thus, $L$ is big and nef, whence $h^i(L+K_{X_d})=h^i(D_k)=0$ for $i=1,2$, and (a) follows by Riemann-Roch.

Assume furthermore that $d+2k \geq 3$. Then $L^2=4(d+2k) \geq 12$. Let 
$Z \subset {X_d}$ be either a base point of $|D_k|=|L+K_{X_d}|$ or a  length-two scheme that is not separated by $|D_k|$. Reider's theorem yields the existence of an effective divisor $F$ on ${X_d}$ containing $Z$ and satisfying $(F \cdot L,F^2) \in \{(0,-2),(0,-1),(1,-1),(1,0),(2,0)\}$, with $(0,-1)$ and $(1,0)$ being the only admissible cases if $Z$ is a point. 

Assume that $F \cdot A <0$. Then $A$ is a component of $F$ and $A^2<0$. Thus,
$F \sim A + F'$, for some $F'>0$. As $L$ is nef, we have
\[ 2 \geq F \cdot L = A \cdot L +F' \cdot L
\geq A \cdot L =   -K_{X_d} \cdot L=2(d+k) \geq 2, \] by our assumptions. Hence, we must have equalities all the way, in particular $d+k=1$ and $F' \cdot L=0$. Hence, $F' \cdot A = F' \cdot B=0$.
Since $0=F^2=A^2+2A \cdot F'+{F'}^2 < {F'}^2$, we get a contradiction by the Hodge index theorem. 

Assume that 
$F \cdot A > 0$. As $F \cdot B \geq 0$ and $k >0$, we must have 
$(F \cdot L,F^2) =(2,0)$ and $F \cdot A =- F \cdot K_{X_d}=1$, a contradiction by  the adjunction formula. 

Hence, $F\cdot A=0$. In particular, $F^2$ must be even by adjunction, so that the only possibilities are $(F \cdot L,F^2) \in \{(0,-2),(1,0),(2,0)\}$.

If $F \cdot B >0$, we have $F \cdot L >0$, whence $F^2=0$. Hence $p_a(F)=1$, so that $F$ has at least one irreducible component that is nonrational. Since $|B|$ is a base-point free pencil, we must therefore have $B \cdot F \geq 2$.
Hence 
\[ 2 \leq F \cdot L =F \cdot A+ k F \cdot B = k F \cdot B,\]
yields the only possibility $k=1$ and $ F \cdot B=2$. The Hodge index theorem yields that
$d=K_{X_d}^2 \leq 0$, whence $d+2k \leq 2$, a contradiction.

Therefore, the only possibility left is $F \cdot B=0$, whence $(F \cdot L,F^2)=(0,-2)$. In particular, $Z$ cannot be a point and we have proved that 
$|D_k|$ is base point free. The nefness of $L$ and $B$ implies that any component  $\Gamma \subseteq F$ must satisfy
$\Gamma \cdot L=\Gamma \cdot B=0$, whence $\Gamma \cdot A=-\Gamma \cdot K_{X_d}=0$. By the Hodge index theorem, we must have ${\Gamma}^2<0$, whence ${\Gamma}^2=-2$ by the adjunction formula. Any such $\Gamma$ must be contained in a member of $|B|$, as $\Gamma \cdot B=0$, whence $\Gamma_0:={\gamma_d}_*\Gamma$ must be contained in a member of the pencil \eqref{eq:pencil2} on $X_4$. The assumptions \eqref{eq:condblow} imply that $\Gamma_0$ cannot be an element of the pencil, whence is contained in a reducible member, thus the same conditions guarantee that $\Gamma_0$ 
does not contain any of the points $u_1,\ldots,u_{4-d}$. Therefore, $\Gamma_0$ is the inverse image of a rational curve $\Gamma_0 \subset X_4$ satisfying
$\Gamma_0^2=-2$. It is well-known and easily checked that any such curve on $X_4$ is one of the $N_1,\ldots,N_4$.

We have therefore shown that $|D_k|$ defines a morphism that is an embedding except for the contraction of $N_1,\ldots,N_4$, proving (b).
\end{proof} 
 
Consider the double cover diagram (cf. Definition \ref{not:cover}):
 \begin{equation}\label{cover1}
 \xymatrix{
 {\widehat X_d}  \ar[r]^{\hat \tau} \ar[d]_{\pi}&   {\widetilde X_d} \ar[d]^{\overline{\pi}} \\
 X_d  \ar[r]^{\tau} & {\overline X_d},}
 \end{equation}
and recall that
$\widetilde{X}_d$ is a half $K3$ surface of degree $2d$ by Lemma \ref{2d}.  
Set $\widehat D_k := \pi^*D_k$ and 
$\widetilde D_k := \hat \tau_* \widehat D_k$.
 It is easy to conclude from Lemma \ref{uno}(b) that, when $d+k \geq 1$ and $d+2k \geq 3$, the divisor  $\widetilde D_k$ is very ample and $$ \dim \vert \widetilde D_k \vert = 6k + 2d. $$

 Finally, let $p: X_d \to \PP^1$ be the conic bundle structure defined by $\vert B \vert$. Since $B\cdot N_i = 0$ for $i = 1 ,\ldots, 4$, then $p$ induces
 a conic bundle structure
 $$
 \overline p: \overline X_d \to \PP^1
 $$
 such that $p = \overline p \circ \tau$.  It is important to remark that $\widetilde X_d$ is defined via
 base change of $\overline p$ as follows. Consider the nonreduced fibers $B_{14} := N_1 + N_4 + 2E_{14}$ and $B_{23} := N_2 + N_3 + 2E_{23}$ 
of $p$
and let $ \beta: \PP^1 \to \PP^1 $ 
be the double cover branched along $p(B_{14})+p(B_{23})$. Then   
$$
 \xymatrix{
 {\widetilde X_d} \ar[r]^{\overline \pi} \ar[d]_{\tilde p} &  {\overline X_d} \ar[d]^{\overline p}\\
{\PP^1} \ar[r]^{\beta} & {\PP^1}
}
 $$
is a Cartesian square, where $\tilde p: \widetilde X_d \to \PP^1$ is the conic bundle structure induced by the pullback of the conic bundle structure $\overline p$. Denote by $\widetilde B$ a fiber of $\tilde p$. One has $2 \widetilde B \sim \overline{\pi}^*\tau_*B \sim \hat{\tau}_*\pi^*B
$ 
and it is easy to see that $|\widetilde{B}|$ is a pencil of conics on 
$\widetilde{X}_d$. Moreover,
recalling that $\widetilde{A}$ is an anticanonical divisor on $\widetilde{X}_d$, we have
\begin{equation}\label{tilde}
\widetilde{D}_k \sim -K_{\widetilde{X}_d} + 2k\widetilde{B} \sim \widetilde{A}+2k\widetilde{B}.
\end{equation}

The following important properties are  immediate:

\begin{lemma} \label{12345}
  Any smooth curve $\widetilde{D} \in |\widetilde{D}_k|$ is hyperelliptic of genus $2k+1$, with its $g^1_2$  cut out by $|\widetilde{B}|$.
Furthermore, for any smooth $D \in |D_k|$, let $\widetilde{D}:=\overline{\pi}^{-1}(\tau_*D) \in |\widetilde{D}_k|$. Then 
$$
\pi_D:=\overline{\pi}|_{\widetilde D}: \widetilde{D} \longrightarrow \tau_*D \cong D
$$
is the \'etale double covering induced by $\eta_D$ (cf. Lemma \ref{lemma:prymhyp}). 
\end{lemma}

 \section{A partial compactification of $\F_g^{\n,s}$ and $\F_g^{\n, ns}$  by unions of untwisted half Nikulin surfaces} \label{sec:SS}

In this section we construct type II $K3$ surfaces that are limits of polarized Nikulin surfaces by gluing pairs of half-Nikulin surfaces of the same degree $d$ with $1 \leq d \leq 4$. 

Let $X_d$ be a half-Nikulin surface as in \S \ref{sec:const-unt}, keeping the notation therein. Let $A$ be a fixed smooth anticanonical divisor on $X_d$, recalling  Remark \ref{rem:startA} and in particular \eqref{doppi}. 

We make the same construction starting from the same $A$ but (possibly) different points $x'\in A$ and $u_1',\ldots,u_{4-d}'\in A$, thus obtaining another half-Nikulin surface $X'_d$ of degree $d:=K_{X'_d}^2$; we use the notation $N_{ij}'$, $e_{ij}'$, $u_k'$, $\Delta'$ for the obvious objects on $X_d'$. Then,  the following relation on $A$ holds:
\begin{equation}\label{primi}
e_{13}+e_{14}+e_{23}+e_{24}\sim e_{13}'+e_{14}'+e_{23}'+e_{24}'.
\end{equation}

Let $S_d := X_d \sqcup_{A} X'_d$ be the surface obtained by gluing $X_d$ and $X'_d$ along $A$; then, $S_d$ has simple normal crossing singularities and trivial canonical bundle. 

\begin{lemma}\label{incollamento}
The following hold:
\begin{itemize}
\item[(i) ]the pair $(\Delta,\Delta')$ defines a Cartier divisor $M$ on $S_d$;
\item[(ii)] For any integer $k\geq 1$ the pair $(D_k,D_k')$ defines a Cartier divisor on $S_d$ if and only if the following equivalence of divisors on $A$ is satisfied:
    \begin{equation}\label{eq:condkB}
    2ke_{14}-\sum_{j=1}^{4-d}u_{j} \sim 2ke_{14}'-\sum_{j=1}^{4-d}u_{j}'. 
  \end{equation}
  \end{itemize}
  \end{lemma}
  
\begin{proof}
Here, and in the sequel, we use the fact that a Cartier divisor on $S_d$ is a pair of divisors on its two components with matching intersection along the double curve. Then, point (i) follows from \eqref{n2} and \eqref{primi}, whereas (ii) follows from \eqref{div0}.
  \end{proof}

When \eqref{eq:condkB} is satisfied, let $H_{d,k}\in\Pic(S_d)$ denote the Cartier divisor defined by $(D_k,D_k')$ and $\varphi_{d,k}$ be the map induced by the complete linear system $|H_{d,k}|$.

\begin{thm} \label{emb}
The map $\varphi_{d,k}$ is an embedding except for the  contraction of the eight curves $N_{i}$ and $N_i'$ to ordinary rational double points. Its image is the (transversal) gluing $\overline{S}_d:=\overline{X}_d\sqcup_{A} \overline{X}_d'\subset\mathbb{P}^{4k+d+1}$ of $\overline{X}_d$ and $\overline{X}_d'$ along $A$, which is embedded as an elliptic normal curve of degree $2k+d$ in $\mathbb{P}^{2k+d-1}$.

A general $C \in  |H_{d,k}|$ is the union of two smooth irreducible curves $D \subset X_d$ and $D'\subset X_d'$  of genus $k+1$, intersecting transversely in $d+2k $ points. In particular, 
$ p_a(C)= 4k+d+1$. Moreover, $D$  and $D'$ are hyperelliptic and their $g^1_2$ is cut out by $\vert B \vert $ and $ \vert B' \vert $, respectively.T he pair $(C,  \mathcal O_C(M))$ is a stable Prym curve.
\end{thm}

\begin{proof}
Apply Lemma \ref{uno} along with \eqref{eq:tre} to both $D$ and $D'$.  
From   
\begin{equation} \label{eq:dec}
 \xymatrix{ 0 \ar[r] & \O_{X'_d}(kB') \ar[r] & H_{d,k} \ar[r] & \O_{X_d}(D_k) \ar[r] & 0,}
\end{equation}
 we obtain $h^0(H_{d,k})=4k+d+2$,
because $h^1(\O_{X'_d}(kB'))=0$ and $h^0(\O_{X'_d}(kB'))=k+1$.
Furthermore, from \eqref{eq:dec} and
\[
 \xymatrix{ 0 \ar[r] & \O_{X_d}(kB) \ar[r] & \O_{X_d}(D_k) \ar[r] &  \O_{A}(D_k) \ar[r] & 0}
\]
we get  $h^0(A, H_{d,k}|_{A})=2k+d$ and obtain  surjective restriction maps
\[
 \xymatrix{ H^0(S_d, H_{d,k}) \ar@{->>}[r] & H^0(X_d, H_{d,k}|_{X_d}) \ar@{->>}[r] & H^0(A, H_{d,k}|_{A}).}
\]

We next prove that  $\varphi_{d,k}$ is an embedding outside the contracted curves, which implies  $\varphi_{d,k}(S_d) = \overline S_d $. 

By Lemma \ref{uno}, $\varphi_{d,k}$ is an embedding outside the contracted curves on each of the components of $S_d$.
Suppose that there are points $y \in X_d\setminus A$ and $y '\in X_d'\setminus A$, such that $\varphi_{d,k}(y)=\varphi_{d,k}(y')$.  Let $C \subset X_d$ be a general curve in $|D_k|$ passing through $y$. Then $C$ intersects $A$ transversely along a divisor $Z \in |\O_A(D_k)|=|\O_A(D_k')|$. The ideal sequence of $Z \subset A \subset X_d'$ tensored by $\O_{X_d'}(D_k')$:
\[\xymatrix{ 0 \ar[r] & \O_{X_d'}(kB') \ar[r] & \O_{X_d'}(D_k') \* \I_{Z/X_d'} \ar[r] &  \O_{A} \ar[r] & 0,}\]
and the vanishing $h^1(\O_{X_d'}(kB'))=0$,  prove  that $|\O_{X_d'}(D_k' )\* \I_{Z/X_d'}|$ is base point free off $Z$, so that we can find a $C' \in |D_k'|$ not passing through $y'$ and such that $C \cap C'=Z$. Thus we have found a member $C \cup C'\in|H_{d,k}|$ containing $y$ but not $y'$, a contradiction.

We argue similarly for any two infinitely near points $y$ and $y'$ on $A$.

This proves that $\varphi_{d,k}$ is an embedding outside the contracted curves. 

The statements about a general $C\in |H_{d,k}|$ easily follow from Lemma \ref{lemma:prymhyp}.
\end{proof}

Choose other $2d$ points $u_{4-d+1},\ldots,u_{4}, u_{4-d+1}',\ldots,u_{4}'$ on $A$ 
 satisfying:
\begin{equation}\label{stabile}
u_1+u_2+u_3+u_4+u_1'+u_2'+u_3'+u_4'\in |L_1^{\otimes2}\otimes L_2^{\otimes 2}|
\end{equation}
(recalling Remark \ref{rem:startA}) and denote by $X$ (respectively, $X'$) the blow up of $X_d$ (resp., $X'_d$) at $u_{4-d+1},\ldots,u_{4}$ (resp. $u_{4-d+1}',\ldots,u_{4}'$). By abuse of notation, we denote the strict transform on $X$ (or $X'$) of a divisor on $X_d$ (or $X_d'$) still by the same name. 

\begin{lemma} \label{ss}
The surface $S:= X \sqcup_{A} X'$ is a stable $K3$ surface of type II, whose construction depends on $10$ parameters.
\end{lemma}

\begin{proof}
Just use \eqref{stabile} in order to show that $\N_{A/X}\otimes \N_{A/X'}\simeq \mathcal{O}_A$. 
Concerning the number of moduli of $S$, recalling Remark \ref{rem:startA}, there is one for the choice of the elliptic curve $A$, one for the choice of the line bundles $L_1,L_2\in \Pic^2(A)$ such that $L_1^{\*2} \cong L_2^{\*2}$, two for the choice of the points $x,x'\in A$ defining the two configurations of four lines, and seven for the choice of the eight points $u_i$ and $u_i'$ satisfying \eqref{stabile}. Then one has to subtract one parameter because of the automorphism group $\Aut(E)$.
\end{proof}

 We denote by $f:S\to S_d$ the natural contraction map and still by $M$ (respectively, $N_i,N_i'$) the pullback under $f$ of $M$ (and of $N_i,N_i'$, respectively). 

Fix henceforth integers $k$ and $d$ such that
\[
 k \geq 1, \; \; d \in\{1,2,3,4\}.
\]
and let $g:=4k+d+1$. Consider the Cartier divisor $H:=f^*H_{d,k}$ on $S$, defining a linear system of genus $g$ curves. Recall that, in addition to the line bundles $H$, $M$, $N_{i}$ and $N_i'$, the surface $S$ carries the Cartier divisor $\xi$ in \eqref{eq:canxi}. 
We  denote by $\Lambda\subset \Pic(S)/\langle \xi \rangle$ the lattice generated by $H,M,N_1,N_2,N_3,N_4,N_1',N_2',N_3',N_4'$ modulo $\xi$.

\begin{proposition} \label{prop:embnik0}
The embedding $\Lambda \subset \Pic S/\langle \xi \rangle$ is primitive except precisely when $d=4$ (thus, $g=4k+5$) and
\begin{equation}
    \label{eq:nonprim}
    ke_{14} \sim ke_{14}' \; 
\;  \mbox{on} \; \; A. 
  \end{equation}
In this case $H$ is primitive and the sum of $H$ and four of the $(-2)$-curves is $2$-divisible in $\Pic S$.
\end{proposition} 
\begin{proof}
We denote by $U_i$ the exceptional divisor over $u_i$, $i = 1, \dots, 4$. By abuse of notation, we still denote by $H$ a divisor in the linear system $|H|$. If the embedding $\Lambda \subset \Pic S/\langle \xi \rangle$ is not primitive, there exist integers $\beta_i,\beta_i'$ for $1\leq i\leq 4$ such that $H + \sum_{i=1}^4 \frac{\beta_i}{2} N_i+ \sum_{i=1}^4 \frac{\beta_i'}{2} N_i'  + \varepsilon \xi$ 
 is $m$-divisible in 
$\Pic S$, for some $m \geq 2$, and $\varepsilon =0$ or $1$. In particular, its restriction to both $X$ and $X'$ is $m$-divisible.
Its intersection with $U_4$ equals $-\varepsilon$, whence we must have $\varepsilon=0$. If $d <4$, then $U_1\cdot H|_X=1$ and $U_1\cdot N_i=0$ for $1\leq  i\leq 4$, whence  the embedding is primitive. Assume now $d=4$. Since $B \cdot H|_{X}=2$, the only possibility is $m=2$. 
By \eqref{div0}, the  restriction of $H|_X+\sum_{i=1}^4 \frac{\beta_{i}}{2}N_{i}$ to $A$ lies in $|L_1^{\otimes 2}(2ke_{14})|$. Analogously, the restriction of $H|_{X'}+\sum_{i=1}^4 \frac{\beta_{i}'}{2}N_{i}'$ to $A$ lies in $|L_1^{\otimes 2}(2ke_{14}')|$, whence the $2$-divisibility of
$H + \sum_{i=1}^4 \frac{\beta_i}{2} N_i+ \sum_{i=1}^4 \frac{\beta_i'}{2} N_i'$ in 
$\Pic S$ requires \eqref{eq:nonprim} to be satisfied. 

Conversely, assume that \eqref{eq:nonprim} holds. 

If $k$ is even, then \eqref{div0} implies that $H|_{X}+N_{3}+N_{4}$ is $2$-divisible in $\Pic X$, with its half given by
\[
w:= \ell_1+\frac{k}{2}N_{1}+N_{3}+\frac{k+2}{2}N_{4}+kE_{14}
\]
(recalling \eqref{eq:ell}).
The analogous equation holds for $w':=(H|_{X'}+N_{3}'+N_{4}')/ 2$ on $X'$.
Condition \eqref{eq:nonprim} implies that $w|_A \sim w'|_A$, so that
$(w,w')$ is an element in $\Pic S$. It follows that $H+N_{3}+N_{4}+N_{3}'+N_{4}'$ is $2$-divisible in $\Pic R$.

If $k$ is odd, then
 \eqref{div0} yields that $H_{|X}+N_{1}+N_{3}$ is $2$-divisible in $\Pic X$, with its half given by
\[ w:=\ell_1 +\frac{k+1}{2}N_{1}+N_{3}+\frac{k+1}{2}N_{4}+kE_{14},\]
and analogously one defines $w'$ on $X'$.
As above,
$(w,w')$ defines an element in $\Pic S$ and thus $H+N_{1}+N_{3}+N_{1}'+N_{3}'$ is $2$-divisible in $\Pic S$.
\end{proof}

    \begin{remark}\label{duetipi}
 Since condition \eqref{eq:nonprim} is exactly \eqref{eq:condkB} divided by two, we can always impose that $(D_k,D_k')$ is Cartier on $S_d$ but \eqref{eq:nonprim} is not satisfied, so that 
the embedding $\Lambda \subset \Pic S/\langle \xi \rangle$ is primitive. 

On the other hand, if \eqref{eq:nonprim} is satisfied, then $(D_k,D_k')$ automatically defines a Cartier divisor on $S_d$ by Lemma \ref{incollamento}. Therefore, for any $g \equiv 1 \; \mod 4$, we can also construct surfaces $S$ such that the embedding $\Lambda \subset \Pic S/\langle \xi \rangle$ fails to be primitive; the construction still depends on $10$ parameters.
\end{remark}

By Lemma \ref{ss} and \cite[Thm. 5.10]{fri}, the surface $S$ is smoothable and its versal deformation space   has  a unique smoothing component $\V$, which is $20$-dimensional and contains a smooth divisor $\D$ such that 
\begin{itemize}
\item $\V\setminus\D $ parametrizes smooth $K3$ surfaces;
 \item $\D$ parametrizes locally trivial deformations of $S$ that remain stable.
\end{itemize}

\begin{proposition}\label{pm}
A general point in $\D$ parametrizes a birational modification along a divisor in $|T^1_R|$ of a surface $R$ obtained by gluing  two half $K3$ surfaces of degree $d$ along an anticanonical divisor.
\end{proposition}
\begin{proof}
It is enough to show that surfaces as in the statement depend on  $19$ parameters. Degree $d$ Del Pezzo surfaces have $10-2d$ moduli and the anticanonical 
linear system on such a surface is $d$-dimensional. In order to glue two Del Pezzo surfaces $W$ and $W'$  along an anticanonical divisor, the two anticanonical divisors should be isomorphic and this imposes one condition. However, given two copies of the same elliptic curve, one can glue them in infinitely many ways thanks to the one-dimensional family of automorphisms of the curve. Therefore, such a gluing $R=W \sqcup_{A} W'$ depends on $20-2d$ parameters. Since $T^1_R\simeq \N_{A/W}\otimes \N_{A/W'}$ has degree $2d$, then the linear system $|T^1_R|$ has dimension $2d-1$, and hence altogether the number of moduli is $19$.  
\end{proof}

Define the polarization
\[ L:= H (-M)  \in \Pic  S.\]
By  \cite[Prop. 4.3]{fri2}, there are unique smooth $19$-dimensional subvarieties $\V_{g-2}, \V_g \subset \V$ such that
\begin{itemize}
\item $\D_g:=\D \cap \V_g$ and $\D_{g-2}:= \D \cap \V_{g-2}$ are smooth of dimension $18$, contain the point $[S]$ and parametrize deformations in $\D$ preserving the polarizations $H$ and $L$, respectively;
\item $\V_g\setminus\D_g$ (respectively, $\V_{g-2} \setminus \D_{g-2}$) parametrizes smooth polarized $K3$ surfaces of genus $g$ (resp., $g-2$).
\end{itemize}

Let $\V_i$ and $\V_i'$ denote the loci  in $\V_g$  containing a deformation of $N_i$ and $N_i'$, respectively.

\begin{lemma} \label{lemma:ck}
  For $1\leq i\leq 4$, the loci $\V_i$ and $\V_i'$ are divisors in $\V_g$ that are smooth at the point $[S]$ and do not lie inside $\D_g$.
\end{lemma}

\begin{proof}
Since  $N_i$ and $N_i'$ satisfy $N_i \cdot K_S=N_i' \cdot K_S= 
0$ and stay off the singular locus of $S$, the argument is identical to the one at the beginning of \cite[Pf. of Lemma 4.1]{ck}.
\end{proof}

\begin{proposition} \label{prop:locus11}
The locus 
\[ \N'_g:= \V_g  \cap \V_{g-2} \cap \V_1 \cap \V_2 \cap \V_3 \cap\V_4\cap\V_1'\cap\V_2'  \cap \V_3'\]
has an $11$-dimensional component $\N_g$ whose general point parametrizes a smooth primi\-tively polarized Nikulin surface of genus $g$. Its type is non-standard if and only if $g\equiv 1\, \mathrm{mod}\,4$ and equation \eqref{eq:nonprim} holds.
\end{proposition}
\begin{proof}
By Lemma  \ref{lemma:ck}, $\N'_g$ contains $[S]$ and has dimension $\geq 19-8=11$.  
Locally around the point $[S]$, its intersection with $\D$ consists of deformations of $S$ as in Proposition \ref{pm} that preserve $H$, $L$, $N_1,N_2,N_3,N_4,N_1',N_2',N_3'$, whence, automatically, also
$N_4'$. By Proposition \ref{pm} and Lemma \ref{ss}, this locus is $10$-dimensional,whence $\N'_g$ does not lie entirely inside $\D_g$.  The point  $[S]$ belongs to an $11$-dimensional component $\N_g$ of $\N_g'$, whose general point corresponds to a smooth $K3$ surface $Y$ that carries two polarizations $H^Y,L^Y$ of genus $g$ and $g-2$ respectively, and $7$ mutually disjoint rational curves $N_1^Y,\ldots,N_7^Y$, not intersecting $H_Y$ and each intersecting $L_Y$ in exactly one point. The surface $Y$ contains one more $(-2)$-curve with the same properties, whose class is given by 
$2(H_Y-L_Y)-\sum_{i=1}^7N_i^Y$. This curve degenerates to $N_4'$ as $Y$ degenerates to $S$, whence 
it is smooth, rational and disjoint from the first seven. The sum of the eight disjoint rational curves is linearly equivalent to $2(H^Y-L^Y)$ and hence $2$-divisible. As a consequence, setting $M^Y:=H^Y\otimes (L^Y)^\vee$, the triple $(Y,M^Y,H^Y)$ is a Nikulin surface of genus $g$. To determine its type  and conclude that $H^Y$ is primitive, it is enough to apply Proposition \ref{prop:embnik0} and Remark \ref{duetipi}.
\end{proof}

There are dominant maps $\V_g \setminus \D_g \to \F_g$ and $\N_{g}\setminus (\N_{g} \cap \D_{g}) \longrightarrow \F_{g}^{\n,s}$  
(or to ${\F_{g}^{\n,ns}}$, if $g\equiv 1\, \mathrm{mod}\,4$ and equation \eqref{eq:nonprim} holds). By Propositions \ref{pm} and \ref{prop:locus11}
one may use the boundary divisors $\D_{g} \subset \V_{g}$ and 
$\N_{g} \cap \D_{g} \subset \N_{g}$ to obtain partial compactifications of 
$ \F_{g}$ and  $\F_{g}^{\n,s}$  (or ${\F}_{g}^{\n,ns}$), respectively. More precisely:

\begin{cor} \label{cor:exparcomI} \
\begin{enumerate}
\item  There exists a partial compactification $\overline{\F}_{g}$ of $ \F_{g}$, whose boundary is a smooth irreducible divisor parametrizing isomorphism classes of triples $(R,T,H)$, where  $R$ is the union of two half $K3$ surfaces of degree $d$ glued along a smooth anticanonical divisor,
$T\in |T^1_R|$, and $H$ is a primitive polarization of genus $g$. 
\item
The closure $\overline{\F}_g^{\n,s}$ of $ \F_{g}^{\n,s}$ in $\overline{\F}_{g}$
is a partial compactification of $ \F_{g}^{\n,s}$ whose boundary points represent
unions of  half Nikulin surfaces. The same holds for the closure $\overline{\F}_g^{\n,ns}$ of $ \F_{g}^{\n,ns}$ if $g\equiv 1\, \mathrm{mod}\,4$.
\end{enumerate}
\end{cor}

\begin{proof}
By \cite[Def 4.9 and Thm. 4.10]{fri2} there is a normal separated partial compactification of $\F_g$ obtained by adding  a smooth divisor whose components correspond to various kinds of 
type II degenerations of  $K3$ surfaces. Take the component containing the point $(S_d,u_{4-d+1}+\cdots+u_4+u'_{4-d+1}+\cdots+u'_4, H_{d,k})$, and
hence all surfaces in $\D_g$. Propositions \ref{pm} and \ref{prop:locus11} imply that points in this component  represent isomorphism classes of  triples $(R,T,H)$, with $R$ the transversal union of two half $K3$ surfaces of degree $d$, the line bundle $H\in\Pic(R)$ a  primitive genus $g$ polarization, and $T \in |T^1_{R}|$. 
\end{proof}

Now we consider the double cover of $S_d$ defined by the line bundle $M=(\Delta,\Delta')$ as in Lemma \ref{incollamento}(i). The double cover diagrams \eqref{cover1} for $X_d$ and $X_d'$ extend to:

\begin{equation}\label{diag} \xymatrixcolsep{4pc}\xymatrix{ \widehat X_d\sqcup_{\widetilde A}\widehat X_d' \ar@{=}[r] & 
\widehat{S}_{d} 
\ar[r]^{\hat t=(\hat\tau,\hat\tau')}  \ar[d]_{\Pi=(\pi,\pi')} & \widetilde{S}_{d} \ar@{=}[r] \ar[d]^{\overline{\Pi}=(\overline{\pi},\overline{\pi}')} & \widetilde X_d\sqcup_{\widetilde A}\widetilde X_d' \\
X_d\sqcup_A X_d' \ar@{=}[r] & S_{d}  \ar[r]^{t=(\tau,\tau')} & \overline{S}_{d}\ar@{=}[r] &\overline{X}_d\sqcup_{A} \overline{X}_d',
}
\end{equation}
where $\widetilde{A}$ denotes the curve inverse image of both $A\subset X_d$ and $A\subset X_d'$ under $\pi$ and $\pi'$, respectively. In other words, $\Pi$ is the double cover of $S_d$ branched along the eight ($-2$)-curves $N_i, N_i'$ for $1\leq i\leq 4$, and $\hat t$ is the contraction of their inverse images.

By Lemma \ref{2d}, the surface $\widetilde S_{d}$ is the transversal union of two Del Pezzo surfaces of degree $2d$ along an anticanonical divisor; by Theorem \ref{emb}, it is embedded by $|\overline{\Pi}^*H_{d,k}|$ into $\PP^{\tilde{g}}$, with $\tilde{g}:=2g-1=8k+2d+1$, as a degree $2(\tilde{g}-1)$ surface with trivial canonical bundle. Since $\deg T^1_{\widetilde S_d} = 4d$, the surface $\widetilde S_d$ is not stable. To obtain a stable $K3$ surface, one must as usual perform a birational modification 
 $\widetilde S:=\widetilde X \sqcup_{\widetilde A} \widetilde X'  \to \widetilde S_d=\widetilde X_d\sqcup_{\widetilde A}\widetilde X_d'$ 
along a divisor in $|T^1_{\widetilde S_d}|$, cf. \S \ref{ss:limits}. For instance, one may obtain such an $\widetilde S$ simply by taking the double cover $\widehat S$ of the surface $S$ in Lemma \ref{ss} branched along the eight curves $N_i,N_i'$ for $1\leq i\leq 4$, and then contracting the ramification divisor.

\begin{cor} \label{cor:exparcomsopraI} \
 \begin{enumerate}
 \item  There exists a partial compactification $\overline{\F}_{2g-1}$ of $ \F_{2g-1}$, whose boundary is a smooth irreducible divisor parametrizing isomorphism classes of triples $(\widetilde{R},\widetilde{T},\widetilde{H})$, where $\widetilde{R}$ is the union of two half Del Pezzo surfaces of degree $2d$ glued along a smooth anticanonical divisor,
$\widetilde{T} \in |T^1_{\widetilde{R}}|$, and $\widetilde{H}$ is a primitive polarization of genus $2g-1$. 

\item The locus in $\overline{\F}_{2g-1} \setminus \F_{2g-1}$ of double covers of members
in $\overline{\F}_g^{\n,s} \setminus \F_{g}^{\n,s}$
(or $\overline{\F}_g^{\n,ns} \setminus \F_{g}^{\n,ns}$) form a partial compactification of $\widetilde{\F}_{2g-1}^{s}$
(or of ${\widetilde{\F}_{2g-1}^{ns}}$), cf. \S \ref{ss:strat}.
\end{enumerate}
\end{cor}

\begin{proof}
This is immediate from Corollary \ref{cor:exparcomI}.
\end{proof}

By  \eqref{tilde}, Lemma \ref{12345} and Theorem \ref{emb}, the inverse image of a general curve $C=D\cup D'\in |H_{d,k}|$ under $\Pi$ is a nodal curve $\widetilde{C}=\widetilde{D}\cup\widetilde D'$ lying in $\widehat S_d$ (and also in $\widetilde S_d$), where $\widetilde D$ and $\widetilde D'$ are hyperelliptic curves of genus $2k+1$ intersecting in $2d+4k$ points.

\begin{thm} \label{thm:main-I}
The map $\chi_g^{s}$ is birational onto its image for $g=7$ and $g \geq 10$. The map $\chi_g^{ns}$ is birational onto its image for (odd) genus $g\geq 13$ such that $g \equiv 1 \mod 4$.
\end{thm}

\begin{proof}
We follow the strategy outlined in \S \ref{ss:strat}.  Let $[(S_d,H_{d,k},C)]$ be general in $\P_g^{\n,s}$ or in $\P_g^{\n,ns}$. We consider the surface $\widetilde{S}_d$ in \eqref{diag} containing the double cover $\widetilde{C}$  of $C$. Let $$x:=[\widetilde C\subset \widetilde S_d\subset \PP^{2g-1}]$$ be the point of the flag Hilbert scheme $\mathfrak P_{2g-1}$ determined up to projectivities by the line bundle $\widetilde H:=\O_{\widetilde S_d}(\widetilde C)$. We need to show that, as soon as $(d,k) \not \in \{(1,1),(3,1),(4,1)\}$,  the fiber of $\mathfrak{m}_{2g-1}$ over the point $[\widetilde C\subset \PP^{2g-1}]$ has a component of dimension $2g$, that equals the dimension of the projectivities of $\PP^{2g-1}$ fixing the hyperplane containing $\widetilde C$. Let $\mathfrak{Y}$ be any component of this fiber containing $x$. 

If a general point of  $\mathfrak{Y}$ parametrizes an irreducible surface, then by Lemma \ref{lemma:utilissimo} the $2d+4k$ nodes of $\widetilde{C}$ are contained in the support of a divisor in $|T^1_{\widetilde S_d}|$, which has degree $4d$. Therefore, $2k \leq d$. If equality occurs, then 
$(d,k) \in\{(2,1),(4,2)\}$ and  the nodes of $\widetilde{C}$ define a divisor in $|T^1_{\widetilde S_d}|$, which is the pullback of the divisor in $|T^1_{S_d}|=|\N_{A/X_d} \* \N_{A/X'_d}|$ determined by the 
$d+2k$ nodes on $C$. Hence,
$(-K_{X_d}+kB)|_A \sim (-K_{X_d}-K_{X'_d})|_A$,
which implies the  relation
\[ u_1'+\cdots +u_{4-d}' \sim 2e_{13}+2(1-k)e_{14} \]
on $A$. When $(d,k)=(4,2)$, we get the contradiction $2e_{13} \sim 2e_{14}$. When  $(d,k)=(2,1)$, we obtain the equivalence $u_{1}'+u_{2}' \sim 2e_{13}$ in contradiction with the generality of the points $u_{1}'$ and $u_{2}'$ and thus of $S_d$.
We conclude that $2k<d$, which means $(d,k) \in \{(3,1),(4,1)\}$, but this contradicts our assumption on $g$.

Hence, a general point of $\mathfrak{Y}$ parametrizes a reducible surface. Since embedded unions of half $K3$ surfaces surfaces of degree $2d$ glued along an anticanonical divisor form a dense open subset of an irreducible component of the locus of reducible surfaces in $\H_{2g-1}$, general surfaces in $\mathfrak{Y}$ are of this type. 

 Both components of $\widetilde{C}=\widetilde D\cup\widetilde D'$ are hyperelliptic of genus $2k+1$. Given any surface $W \cup W'$ of the specific type containing $\widetilde{C}$, the Cartier divisor $(\widetilde{D},\widetilde{D}')$ on $W \cup W'$  restricts to the canonical divisor on $\widetilde{C}$. In particular, the normal bundle $\N_{\widetilde{D}/W}$ does not depend on $W$ and has degree
$\widetilde{D}^2=2d+8k$,
as $ \N_{\widetilde{D}/W} \simeq {\omega_{\widetilde{C}}}|_{\widetilde{D}} \simeq
\omega_{\widetilde{D}}(\widetilde{D} \cap \widetilde{D}')$.
Inequality \eqref{eq:condcazzo} reads $d+2k \geq 4$ (as $-K_W$ is ample). Hence, Proposition \ref{prop:recon} ensures that $W$ is determined up to finitely many choices unless
$(d,k)=(1,1)$. Exactly in the same way one reconstructs $W'$ starting from $\widetilde{D}'$ up to finitely many choices.
\end{proof}

\begin{thm} \label{thm:main2-I}
The map $m_g^{\n,s}$ is birational onto its image for $g=13$ and $g \geq 15$. The map $m_g^{\n,ns}$ is birational onto its image for (odd) genus $g\geq 13$ such that $g \equiv 1 \mod 4$.
\end{thm}

\begin{proof}
 Let $[(S_d,H_{d,k},C)]$ be general in $\P_g^{\n,s}$ or in $\P_g^{\n,ns}$
and let $$x:=[C \subset S_d\subset \PP^g=\PP(H^0(H_{d,k})^\vee)]$$ be the corresponding point (up to projectivities) in the flag Hilbert scheme $\mathfrak{P}_{g}$. As outlined in \S \ref{ss:strat}, we show that the fiber of $\f_{g}$ over the point $[C\subset \PP^g]$ has a component of dimension $g+1$, that is, made of projectively equivalent surfaces. 
 
Consider any component 
of that fiber containing $x$. Exactly as in the proof of Theorem \ref{thm:main-I}, one applies Lemma \ref{lemma:utilissimo} to exclude that a general point of $\mathfrak{Y}$ parametrizes an irreducible surfaces except when $(d,k) \in \{(3,1),(4,1)\}$, which yields $g \leq 9$. 

Hence, general points of $\mathfrak{Y}$ represent reducible surfaces and more precisely unions of two half $K3$ surfaces of degree $d$  glued along an anticanonical divisor. We have left only to prove that $C$ lies on finitely many such unions. One may argue as in the proof of Theorem \ref{thm:main-I} and apply Proposition \ref{prop:recon} since both components of $C$ are hyperelliptic. By Remark \ref{rem:puntigenerali}, we may assume condition \eqref{eq:-Kampiuccio} to be fulfilled; it is then enough to  
rewrite condition \eqref{eq:condcazzo} (allowing equality) as $D_k^2 \geq 2g(D_k)+6$. Recalling \eqref{eq:tre}, this becomes
$d+2k \geq 8$, which is satisfied precisely when $g=13$ and $g \geq 15$.
\end{proof}

\section{A change of polarization} \label{sec:parcom2}

We make a birational modification of the limit surfaces  in the previous section and a change of polarization, to treat the non-standard cases for genera $g \equiv 3 \mod 4$.

Let $d=4$ and fix an integer $l \geq 1$. As in \S 4, we construct two weak Del Pezzo surfaces $X_4$ and $X_4'$ of degree $4$ but replace condition \eqref{eq:condkB} with:
\begin{equation} \label{eq:cond7}
 2(l+1)e_{14}  \sim 2le_{14}'+2e_{24}'.
\end{equation}
 We then choose on the elliptic curve $A$ a divisor
\begin{equation}\label{stabile2}
u_1+u_2+u_3+u_4+u_5+u_6+u_7+u_8\in |L_1^{\otimes2}\otimes L_2^{\otimes 2}|,
\end{equation}
and denote by $X$ the blow up of $X_4$ at the eight points $u_{1},\ldots,u_{8}$. As customary, we denote the strict transform on $X$ of a divisor on $X_4$ still by the same name. 

We define $S:= X \sqcup_{A} X_4'$ as the gluing of $X$ and $X_4'$ along $A$.
It is straightforward to check, using \eqref{conics}, \eqref{aggiunta} and \eqref{eq:cond7}, that the pair 
\[
(A+(l+1)B,(l-1)B')
\]
 defines a Cartier divisor on $S$, which we denote by $H$. 
A general curve $C$ in $|H|$ consists of a smooth irreducible component $D \subset X$ of genus $l+2$ intersecting transversely $l-1$ rational curves in $|B'|$,
each in two points,
 whose union we denote by $D'$. As a consequence, $p_a(C)=2l+1$.

Condition \eqref{stabile2} ensures that $T^1_S \cong \O_A$ and the same reasoning as in \S 4 shows that $(S,M,H)$ is a stable limit of smooth polarized Nikulin surfaces of genus $g:=2l+1$ (here, $M$ is defined in the obvious way). More precisely, $(S,M,H)$ moves in a smooth divisor in a partial compactification of either $\F_{g}^{\n,s}$ or $\F_{g}^{\n,ns}$, depending on the following proposition. This compactification is obtained by taking the closure of the moduli space of genus $g$ Nikulin surfaces (of standard or non-standard type) in a partial compactification of $\F_g$, whose boundary points parametrize genus $g$ polarized type II $K3$ surfaces that are unions of two half $K3$ surfaces of degree $-4$ and $4$, respectively.

\begin{proposition} \label{prop:embnik00}
  The embedding $\Lambda \subset \Pic S/\langle \xi \rangle$ is primitive unless
\[
    (l+1)e_{14}  \sim le_{14}'+e_{24}' 
\;  \mbox{on} \; \; A.
\]
In this case $H$ is primitive and either
\begin{itemize}
\item[(i)] $l$ is even (whence $g \equiv 1 \; \mod 4$) and  the sum of $H$ and four of the $(-2)$-curves is $2$-divisible in $\Pic S$; or
\item[(ii)] $l$ is odd (whence $g \equiv 3 \; \mod 4$) and  the sum of $H$ and two of the $(-2)$-curves is $2$-divisible in $\Pic S$.
\end{itemize}
\end{proposition}

\begin{proof}
  The proof is almost identical to that of Proposition \ref{prop:embnik0}. The $2$-divisible linear combinations are
 $H+N_{1}+N_{4}+N_{1}'+N_{3}'$ in case (i) and $H+N_{3}'+N_{4}'$ in case (ii).
\end{proof}

 We consider the double cover of $S$ defined by $M$ and obtain a commutative diagram 
\[ \xymatrix{ \widehat{X} \sqcup_{\widetilde{A}} \widehat{X}'_4  \ar@{=}[r] &  \widehat{S} \ar[r]^{\widehat t}  \ar[d]_{\Pi} & \widetilde{S} \ar[d]^{\overline{\Pi}} \ar@{=}[r] & \widetilde{X} \sqcup_{\widetilde{A}} \widetilde{X}'_4 \\
X\sqcup_A X'_4 \ar@{=}[r]  & S \ar[r]^{t} & \overline{S} \ar@{=}[r] &  \overline{X} \sqcup_A \overline{X}'_4
}\]
similar to \eqref{diag}, where we have used the usual notation. By Lemma \ref{2d}, the surfaces $\widetilde{X}$ and $\widetilde{X}'_4$ are half $K3$ surfaces of degrees $-8$ and $8$, respectively. In particular, one easily proves that $\widetilde{X}'_4 \cong \PP^1 \x \PP^1$.

We denote by $\widetilde{C}:=\widetilde{D}\cup \widetilde{D}'$ the inverse image in $\widetilde{S}$ (or $\widehat{S}$) of  a general curve $C=D\cup D'\in |H|$. Then $\widetilde{D}\in |\widetilde{A}+2(l+1)\widetilde{B}|$,
while $\widetilde{D'}\in |2(l-1)B'|$. 

The stable model $\overline{D}$ of $\widetilde{C}$ in $\overline{\M_g}$ is obtained by contracting $\widetilde{D}'$, whence it is an irreducible nodal curve with precisely $2(l-1)$ nodes, having $\widetilde{D}$ as normalization. Denote by $\nu: \widetilde{D} \to \overline{D}$ its normalization. Then $\N_{\widetilde{D}/\widetilde{X}}\simeq \omega_{\widetilde{D}}(Z) \simeq \nu^*\omega_{\overline{D}}$ , where $Z=\widetilde{D} \cap \widetilde{D}' \subset \widetilde{D}$ is the inverse image of the scheme of nodes of $\overline{D}$.

\begin{thm} \label{thm:main-II}
The maps $\chi_g^{s}$ and $\chi_g^{ns}$ are birational onto their images for 
any odd genus $g \geq 13$. 
\end{thm}

\begin{proof}
  We follow the strategy outlined in \S \ref{ss:strat}. This time we partially compactify $\F_{2g-1}$ by adding a boundary divisor parametrizing stable type II $K3$ surfaces obtained by gluing two half $K3$ surfaces of degree $-8$ and $8$, respectively. In comparison with the proof of Theorem  \ref{thm:main-I}, it  is not necessary to pass to the Hilbert schemes. Indeed, under our assumptions, we are able to prove that 
the fiber  $\bar{f}_{2g-1}^{-1}(\overline{D})$ of the compactified forgetful map is finite, and thus automatically consists of only  one point.

By Lemma \ref{lemma:utilissimo}, in a neighborhood of $[(\widetilde{S},\widetilde{H},\widetilde{C})]$, the fiber of $\bar{f}_{2g-1}^{-1}(\overline{D})$ has a component consisting only of curves on reducible surfaces. Let
$[(R=W_{-8} \sqcup W_8, H^R, C_{-8} \cup C_8)]$ denote a general element in this component. Then $C_{-8} \cong \widetilde{D}$ and its normal bundle $\N_{C_{-8}/W_{-8}} \cong \omega_{\widetilde{D}}(Z)$. Now apply Proposition \ref{prop:recon} to the hyperelliptic curve $C_{-8}$,  which has genus $g(\widetilde{D})=2l+3$, and the line bundle $\N:=\omega_{\widetilde{D}}(Z)$, which has degree $8l$. Condition \eqref{eq:condcazzo} is satisfied if (and only if) $l \geq 6$, that is, $g=g(\overline{D})=2l+1 \geq 13$. Under this assumption the component $W_{-8}$ is determined up to finitely many choices. For each such $W_{-8}$, the anticanonical system has a unique effective member $A^W$, which intersects $\widetilde{D}$ along $Z$. We claim that this also determines $W_8$ in finitely many ways. Indeed, $W_8 \cong \PP^1 \x \PP^1$ and $A^W\subset W_8$ is an anticanonical divisor satisfying
\begin{equation}
  \label{eq:condnorm}
 \N_{A^W/W_8} \cong \N_{A^W/W_{-8}}^{\vee} 
\end{equation}
(as $\N_{A^W/W_{-8}} \* \N_{A^W/W_8} \cong T^1_{R} \simeq \O_{A^W}$). 
The restriction to $A^W$ of one of the rulings of $W_8$ is determined by the pairs of points in the subscheme $Z \subset A^W$. Since $A^W$ is a divisor of type $(2,2)$ on $W_8 \cong \PP^1 \x \PP^1$, condition \eqref{eq:condnorm}  yields only finitely many choices for the restriction of the second ruling of $Y_8$ to $A^W$. Hence, the embedding $A^W \hookrightarrow \PP^1 \x \PP^1$ can be reconstructed in finitely many ways. This proves that also 
$W_{-8} \sqcup W_8$ is determined up to a finite number of choices.
\end{proof}

\begin{thm} \label{thm:main2-II}
The maps $m_g^{\n,s}$ and $m_g^{\n,ns}$ are birational onto their images for any odd genus $g\geq 17$.
\end{thm}

\begin{proof}
  We again follow the strategy outlined in \S \ref{ss:strat} and prove that
the fiber  $\bar{m}_{g}^{-1}(D_0)$ is finite, where $D_0$ is the stable model
of a general curve $C=D \cup D' \in |H|$, that is, $D_0$ is obtained by contracting $D'$ and identifying $l-1$ pairs of points on $D$. Again by Lemma \ref{lemma:utilissimo}, it is enough to show that there are finitely many unions of half $K3$ surfaces of degrees $-4$ and $4$ containing a curve stably equivalent to $D_0$. One applies  Proposition \ref{prop:recon} to reconstruct, up to finitely many choices, the half $K3$ surface of degree $-4$ containing $D$. Indeed, condition \eqref{eq:condcazzo} reads $D^2 > 2g(D)+6+4$. Recalling that $D^2=4l$ and $g(D)= l+2$, this becomes
$l \geq 8$, that is, $g \geq 17$. Since the reconstructed half $K3$ surface has degree $-4$, it contains a unique anticanonical elliptic curve $A$. Now we want to show that there are only finitely many half $K3$ surfaces of degree $4$ containing $A\cup D'$ with $A$ anticanonical. Let $W$ be any such surface and let $\overline{W}\subset \PP(H^0(\O_W(A))^\vee)$ be its anticanonical model. Recall that $D'$ consists of $l-1$ disjoint rational curves $B_i$ (for $i=1,\ldots, l-1$) each intersecting $A$ transversely in a pair of points forming a divisor in the same $g^1_2$.  Therefore, all curves $B_i$ are members of  the same pencil of conics $|B|$ on $W$. We proceed as in the proof of Proposition \ref{prop:recon} and use the (hyperelliptic) curve $A$ and the pencil  $|B|$ in order to define a $\PP^1$-bundle $R$ and a $\PP^2$-bundle $P$ on $\PP^1$ such that $R$ is a hyperplane section of $P$ and the anticanonical morphism on $W$ factors through a morphism $W\to W^\dagger\subset P$ realizing the image $W^\dagger$ as a conic bundle in $P$; more precisely, one has $W^\dagger\sim 2R$ on $P$ and $A=R\cap W^\dagger$. We still denote by $B_i$ the image of $B_i$ on $W^\dagger$. It is enough to show that $h^0(\I_{A\cup B_1\cup\cdots\cup B_{l-1}/P}(W^\dagger))=1$. Since $A\cup B_1\cup\cdots\cup B_{l-1}=(R+F_1+\cdots+F_{l-1})\cap W^\dagger$, where $F_i$ is the $\PP^2$-fiber of $P$ containing $B_i$, one reduces to proving that $h^0(\O_P(W^\dagger-R-(l-1)F))=0$ with $F$ the class of the divisors $F_i$. By using the equivalence $W^\dagger\sim 2R$ and then restricting to $W^\dagger$, it suffices to check that $ A-(l-1)B$ is not effective on $W$. This holds true because its intersection with $A$ is negative as soon as $l\geq 4$.
\end{proof}

\section{A partial compactification of $\F_g^{\n,s}$ and $\F_g^{\n,ns}$  by unions of twisted half Nikulin surfaces} \label{sec:parcom}

\subsection{Construction of twisted half Nikulin surfaces} \label{sec:const-tw}

Let $g \geq 3$ be an odd integer and $\PP^1 \x \PP^1 \cong R \subset \PP^g$ be a rational normal scroll of degree $g-1$. Fix a smooth anticanonical divisor $A_R \in |-K_{R}|$; thus, $A_R \subset \PP^g$ is an elliptic normal curve of degree $g+1$. 
 We set  $n:= \frac{g-1}{2}$ and denote by $\s$ and $\f$ a minimal section and a fiber of the ruling of $R$, respectively, so that 
$A_R \sim -K_{R} \sim 2(\s+\f)$ and
\[
D_n:=\O_{R}(1) \simeq \O_{R}(\s + n\f).
\]

Pick four distinct points $x_1,\ldots, x_4$ on $A_R$
and blow up $R$ first at these points, and then at the intersection points of the four exceptional curves $N_1,\ldots,N_4$ with the strict transform of $A_R$.
Let 
\[ \xymatrix{\tau_R: X \ar[r] & R}\]
be the resulting  sequence of eight blow ups. 
On $X$ we denote by $A$ the strict transform of $A_R$, still by $N_i$ the strict transform of the curve $N_i$ for $i=1,\ldots, 4$, and by $E_1,\ldots, E_4$  the exceptional curves obtained at the second set of blow ups. Then $N_i^2=-2$, $E_i^2=-1$ and one has:
\[  A \sim  {\tau_R}^*A_R- \sum_{i=1}^4 N_i - 2 \sum_{i=1}^4 E_i,\]
so that $ A$ is anticanonical on $X$ and $ A^2=0$. Furthermore, the divisor
\begin{equation}
  \label{eq:div2}
   A + \sum_{i=1}^4 N_i \sim 2\left[{\tau_R}^*(\s+\f)-\sum_{i=1}^4 E_i\right]
\end{equation}
is $2$-divisible on $X$ and $A^2=0$. In particular, $X$ is a half Nikulin surface of twisted  type and degree $0$. The double cover diagram \eqref{cover} associated with $X$
fits into 
\begin{equation}\label{cover2}
 \xymatrix{
 {\widehat X} \ar[r]^{\hat \tau} \ar[d]_{\pi} &   {\widetilde X} \ar[r]^{\widetilde \tau} \ar[d]_{\overline\pi}&   {\widetilde R}\ar[dd]^{\widetilde \pi}\\
 X \ar[drr]_{\tau_R} \ar[r]^{\tau} & {\overline X} &  \\
& & { R};
}
  \end{equation}
here, as usual, $\pi$ is the double covering branched on $A + \sum_{i=1}^4 N_i$,
while $\tau$ is the contraction of the $N_i$ and $\hat{\tau}$ of their inverse images $\pi^{-1}(N_i)$ on $\widehat{X}$. The curves $\pi^{-1}(E_i)$ are $(-2)$-curves doubly covering $E_i$ and intersecting $\pi^{-1}(N_i)$ in one point. Thus, their images in $\widetilde{X}$ are $(-1)$-curves and $\widetilde{\tau}$ is their contraction to a smooth surface $\widetilde{R}$. The map $\widetilde \pi$ is the  double cover of $R$ branched along $A_R$. 
 Denoting by $\widetilde{A}$ the inverse image of $A_R$ on $\widetilde{R}$, we have
$2\widetilde{A} \sim \widetilde{\pi}^*A\sim 2 \widetilde{\pi}^*(\s+\f) $, so that 
$ \widetilde{A} \sim \widetilde{\pi}^*(\s+\f)$ and 
\[ K_{\widetilde{R}} \sim \widetilde{\pi}^*K_{R}+\widetilde{A} \sim
\widetilde{\pi}^*(-2\s-2\f+\s+\f)\sim -\widetilde{A}.\]
In particular, $\widetilde{A}$ is anticanonical and $\widetilde{A}^2 =2(\s+\f)^2=4$. Moreover, one can show that $\widetilde{A}$ is ample and the divisor
$\widetilde{B}:= \widetilde{\pi}^*\f$ satisfies $\widetilde{B}^2=0$ and
$\widetilde{B} \cdot \widetilde{A}=\widetilde{\pi}^*\f \cdot \widetilde{\pi}^*(\s+\f)=2$. To summarize, we have:

\begin{lemma}\label{gattoboy}
The surface $\widetilde{R}$ is a Del Pezzo surface  of degree $4$ endowed with a pencil of conics $\widetilde{B}$. 
\end{lemma}

Finally, we set $\widetilde{D}_n:=\widetilde{\pi}^*D_n$, so that
\begin{equation}
  \label{eq:sezipsop}
  \widetilde{D}_n \sim \widetilde{A}+(n-1)\widetilde{B}.
\end{equation}
and $\dim | \widetilde{D}_n|=3n+1$. Clearly, any smooth curve in $|\widetilde{D}_n|$ is hyperelliptic of genus $n$, with its $g^1_2$ being cut out by $|\widetilde{B}|$.

\subsection{Limit surfaces} \label{ss:parcom}

As in \S \ref{sec:SS}, we construct a limit of polarized Nikulin surfaces as follows. The rational normal scroll $ R \subset \PP^g$ can be constructed starting with the elliptic normal curve $A_R \subset \PP^g$ and choosing a  {\it general} line bundle $ \L  \in \Pic^2(A) $: indeed, the surface $ R \subset \PP^g$ is the union of 
the secant lines to $A_R$ spanned by the divisors in $| \L |$. Choose another general line bundle $ \L ' \in \Pic^2(A) $ such that $ \L' \not\simeq  \L $ and let $ R' \subset \PP^g$ be the associated rational normal scroll; again one has $ R' \cong \PP^1 \x \PP^1$. 
Divisors on $R'$ are denoted in the same way as those on $R$ adding a $'$.

By \cite[Thm.\;1]{clm}, $R$ and $R'$  intersect transversely along $A_R$, which is anticanonical on both scrolls. Hence, the surface $W:=R \sqcup_{A_R} R'$ has normal crossing singularities and $\omega_{W}$ is trivial. Moreover, by \cite[Thm.\;3]{clm}, the surface $W$ is a flat limit of smooth $K3$ surfaces in $\PP^g$. We have
\[ T^1_{W} \cong \N_{A_R/R} \* \N_{A_R/R'} \cong \O_{A_R} (\s+\f+\s'+\f')^{\* 2}.\]

We repeat the same procedure (and same notation) used while constructing $X$ starting from $R$, in order to obtain a surface $X'$ that is the blow-up of $R'$ along four pairs of infinitely near points on $A_R$. In particular, the four  points on $A_R$ blown up at the first step are denoted by $x'_1,\ldots, x'_4$ and chosen so that
\begin{equation}
  \label{eq:cond}
  x_1+\cdots+  x_4+x'_1+\cdots+x'_4 \in |\O_{A_R}(\s+\f+\s'+\f')|
\end{equation}
We consider the reducible surface $S := X \sqcup_{A} X'$. The analogue of Lemma \ref{ss} holds:
\begin{lemma} \label{ssdispari}
The surface $S$ is a stable $K3$ surface of type II, whose construction depends on $10$ parameters.
\end{lemma}

\begin{proof}
Condition \eqref{eq:cond} implies that
$\N_{A/X} \* \N_{A/X'} \cong \O_A$. 
Concerning the number of moduli of $S$, there is one for the choice of the elliptic normal curve $A_R\subset \mathbb{P}^g$, two for the choices of the line bundles $ \L, \L'\in \Pic^2(A)$ and seven for the choice of the eight points $x_i$ and $x_i'$ satisfying \eqref{eq:cond}.
\end{proof}

 We define 
$H=H_n$ as the line bundle on $S$ determined by the pair  of Cartier divisors
$({\tau_R}^*D_n,{\tau_{R'}}^*D'_n)$ on $X$ and $X'$, respectively. 
The eight smooth rational curves $N_1,\ldots, N_4,$ $N'_1,\ldots,N'_4$ also define Cartier divisors on $S$, since they do not intersect the singular locus of $S$. Finally, recalling \eqref{eq:canxi},  equivalence \eqref{eq:div2} and its analogue on $X'$ ensure the existence of a line bundle $M$ on $S$ satisfying 
\[
M^{\* 2}_{|X} \cong \O_{R}\left(\sum_{i=1}^4 N_i\right) \*  
\xi_{|X} \; \; \mbox{and} \; \; M^{\* 2}_{|X'} \cong \O_{X'}\left(\sum_{i=1}^4 N'_i\right) \* 
\xi_{|X'}.\]
Let  $\Lambda\subset \Pic(S)/\langle\xi\rangle$ be the lattice generated by $H,M,N_1,N_2,N_3,N_4,N_1',N_2',N_3',N_4'$ modulo $\xi$.

\begin{proposition} \label{prop:embnik0dispari}
The embedding $\Lambda \subset \Pic S/\langle\xi\rangle $ is primitive.
\end{proposition}

\begin{proof}
If not, there exist integers $\beta_i,\beta_i'$ for $1\leq i\leq 4$ such that $H + \sum_{i=1}^4 \frac{\beta_i}{2} N_i+ \sum_{i=1}^4 \frac{\beta_i'}{2} N_i'+\varepsilon\xi$ is $m$-divisible in 
$\Pic S$, for some $m \geq 2$, and for $\varepsilon=0$ or $1$. In particular, its restrictions to $X$ and $X'$ are $m$-divisible. But this is impossible, as their intersections with ${\tau_{R}}^*\f$ and $(\tau_{R'})^*\f'$ are
$1+2\varepsilon$ and $1-2\varepsilon$, respectively.
\end{proof}

By Lemma \ref{ssdispari}, Proposition \ref{prop:embnik0dispari}  and \cite[Thm. 5.10]{fri}, arguing as in \S \ref{sec:SS}, the surface $S$ is smoothable to a primitively polarized Nikulin surface  of standard type. In particular, one constructs  
 a partial compactification $\overline{\F}_{g}^{\n,s}$ of $ \F_{g}^{\n,s}$
whose boundary points parametrize reducible surfaces that are unions of two half Nikulin surfaces of twisted type and degree $0$ like $X$ and $X'$.

 Diagrams \eqref{cover2} for $X$ and $X'$ give rise to a diagram 
 \[ 
\xymatrixcolsep{4pc}\xymatrix{ \widehat{S} =\widehat X \sqcup_{\widehat A}\widehat X' \ar[r]  \ar[d] & \widetilde{S}=\widetilde X\sqcup_{\widetilde A}\widetilde X'
\ar[r] \ar[d]  & \widetilde{W}= \widetilde R \sqcup_{\widetilde A}\widetilde R' \ar[dd]\\
S=X \sqcup_A X'   \ar[r] \ar[drr] & \overline{S}=\overline{X}\sqcup_A\overline{X}' & \\ 
& & W= R \sqcup_{A_R} R'.
}
\]
By Lemma \ref{gattoboy}, the surface $\widetilde W$ is the transversal union of two degree four Del Pezzo surfaces  along an anticanonical divisor. The line bundle $\widetilde H$ on $\widetilde W$ defined by the pair
\[ \left(\widetilde{D}_n,\widetilde{D}'_n\right) =\left(\widetilde{A}+(n-1)\widetilde{B},\widetilde{A}+(n-1)\widetilde{B}'\right),\]
cf. \eqref{eq:sezipsop},
 embeds $\widetilde W$ into $\PP^{2g-1}$ as a degree $4g-4$ surface with trivial canonical bundle. Therefore, $\widetilde S$ represents a point of the component $\mathfrak{H}_{2g-1}$ of the Hilbert scheme of degree $4g-4$ surfaces in $\PP^{2g-1}$ containing smooth primitively embedded $K3$ surfaces of genus $2g-1$. Let $C=D\cup D'$ be a general hyperplane section of $W\subset\PP^g$ and let $\widetilde{C}=
\widetilde{D}\cup\widetilde D'$ be its double cover lying on $\widetilde W\subset\PP^{2g-1}$. Then $\widetilde C$ is nodal, and its components $\widetilde D$ and $\widetilde D'$ are hyperelliptic curves of genus $n$ intersecting  in $2n+2$ points and lying in the linear systems $|\widetilde{D}_n|$ on $\widetilde{R}$ and
$|\widetilde{D}'_n|$ on $\widetilde{R}'$, respectively.

\begin{thm} \label{thm:main-III}
The map $\chi_g^s$ is birational onto its image for any odd genus $g \geq 7$.
\end{thm}

\begin{proof}
 We proceed as in the proof of Theorem \ref{thm:main-I} following the strategy outlined in \S \ref{ss:strat}, and show that the fiber of $\mathfrak{m}_{2g-1}$ over the point $[\widetilde C\subset \PP^{2g-1}]$ has a component of dimension $2g$, that is the dimension of the projectivities fixing the hyperplane containing $\widetilde C$.
 Let $\mathfrak{Y}$ be any component  of that fiber containing the point $[\widetilde C\subset \widetilde{W}\subset \PP^{2g-1}]$.

Assume first that a general element in  $\mathfrak{Y}$ parametrizes an irreducible surface. Then, by Lemma \ref{lemma:utilissimo}, the scheme of $2n+2$ nodes of $\widetilde{C}$ is contained in a divisor in the linear system $|T^1_{\widetilde W}|$ on $\widetilde{A}$, which has degree $8$.  Hence, we obtain $2n+2 \leq 8$, or equivalently, $g \leq 7$. If equality occurs, then $n=3$ and both $\widetilde{C}\subset \widetilde W$  and $C\subset W$ have eight nodes. On $C \subset W$  they must coincide with  
$x_1,\ldots,x_4, x'_1,\ldots,x'_4, $ on $A_R$. Condition \eqref{eq:cond} yields
\[ \O_{A_R}(\s+\f+\s'+\f') \simeq  \O_{A_R}(1) \simeq \O_{A_R}(\s+3\f) \simeq \O_{A_R}(\s'+3\f'),\]
 whence
$\O_{A_R}(2(\f+\f')) \simeq  \O_{A_R}(1)$; 
this provides a nontrivial relation between $ \L $, $ \L'$ and $\O_{A_R}(1)$ on ${A_R}$, thus contradicting the generality of $W$.

Hence, general points of $\mathfrak{Y}$ parametrize unions of Del Pezzo surfaces of degree $4$ glued along an anticanonical divisor. As in the proof of Theorem \ref{thm:main-I}, we apply Proposition \ref{prop:recon} to each of the components of $\widetilde{C}$, which are hyperelliptic of genus $n$ with normal bundle of degree $4n$. Condition \eqref{eq:condcazzo} with ``$\geq$'' reads like $n \geq 3$, that is, $g \geq 7$. 
\end{proof}

\begin{thm} \label{thm:main2-III}
The map $m_g^{\n,s}$ is birational onto its image for any odd genus $g\geq 11$.
\end{thm}

\begin{proof}
  Following the strategy outlined in \S \ref{ss:strat}, it suffices to show that the Gaussian map of a general $[\gamma] \in \im m_g^{\n,s}$  has corank one. But $\gamma$ specializes to a general hyperplane section of $W$, whose Gaussian map has corank one for $g \geq 11$ by \cite{clm}.
\end{proof}

  \end{document}